\documentclass[reqno, 11pt,letter]{amsart}
\allowdisplaybreaks
\usepackage{graphicx,  enumitem, amsmath, amssymb,amsthm,hyperref,amscd,xcolor}
\usepackage{mathabx,manfnt,braket,mathrsfs, comment,mathtools,bm}
\usepackage{multirow}
\usepackage{tikz}
\usepackage{cancel}
\usetikzlibrary{arrows, matrix,decorations.pathmorphing}
\usepackage[utf8]{inputenc}


\newtheorem{thm}{Theorem}[section]
\newtheorem{prop}{Proposition}[section]

\newtheorem{lem}[prop]{Lemma}

\theoremstyle{definition}
\newtheorem{defn}[prop]{Definition}

\numberwithin{equation}{section}

\newcommand{\allof}[1]{{\left(#1\right)}}
\newcommand{\setof}[1]{{\left\{ #1 \right\}}}
\newcommand{\st}{\colon}

\newcommand{\irr}{\mathrm{Irr}}

\DeclareMathOperator{\End}{End}

\newcommand{\norm}[1]{\left\Vert#1\right\Vert}
\newcommand{\abs}[1]{\left\vert#1\right\vert}
\newcommand{\cx}{{\mathbb{C}}}
\newcommand{\C}{{\mathbb{C}}}
\newcommand{\rl}{{\mathbb{R}}}
\newcommand{\hx}{\mathbb{H}}
\newcommand{\Q}{\mathbb{Q}}

\newcommand{\Aa}{\mathfrak{A}}
\newcommand{\Bb}{\mathfrak{B}}
\newcommand{\from}{\colon}
\newcommand{\Span}{\mathrm{span}}

\newcommand{\tmop}[1]{\ensuremath{\operatorname{#1}}}
\renewcommand{\Re}{\tmop{Re}}
\renewcommand{\Im}{\tmop{Im}}

\newcommand{\Z}{\mathbb{Z}}

\newcommand{\ol}{\overline}

\newcommand{\wt}{\widetilde}

\newcommand{\D}{\mathbb{D}}

\DeclareMathOperator{\diag}{diag}

\newcommand\ipr[1]{\left\langle #1 \right\rangle}

\newcommand{\X}{\mathsf{X}}

\title{Polarization identities}
\author{Chase Bender}
\address{Department of Mathematics, University of Notre Dame, Notre Dame, IN 46556,  USA}
\email{cbender3@nd.edu}
\author{Debraj Chakrabarti}
\address{Department of Mathematics, Central Michigan University, Mt. Pleasant,  MI 48859,  USA}
\email{chakr2d@cmich.edu}

\thanks{Debraj Chakrabarti was partially supported by  Simons Foundation Collaboration Grant number 706445.}
\begin{document}
	\begin{abstract} We prove a generalization of the polarization identity of linear algebra expressing the inner product of a complex inner product space in terms of the norm, where the field of scalars is extended to an associative algebra equipped with an involution, and  polarization is viewed as an averaging operation over a compact multiplicative subgroup of the scalars. 
	Using this we prove a general form of the Jordan-von Neumann theorem on characterizing 
	inner product spaces among normed linear spaces, when the scalars are taken in an associative algebra. 
\end{abstract}
\maketitle
\section{Introduction}
 \subsection{Motivation }The goal of this article is to present a general version
 of the  \emph{polarization identity} of elementary linear algebra. This well-known  identity
 takes the form
\begin{equation}
	\label{eq-pol-classical} 
	\ipr{x,y}= \frac{1}{4}\left(\norm{x+y}^2+i\norm{x+iy}^2-\norm{x-y}^2 -i\norm{x-iy}^2\right), \quad i=\sqrt{-1},
\end{equation}
for vectors $x,y$ in a complex inner product space,
where  the  inner product is denoted by $\ipr{\cdot,\cdot}$ and the  induced norm  by $\norm{\cdot}$.
The identity 
has very  important applications in classical Hilbert space theory (e.g. \cite{jvn,halmos}).
This paper takes the point of view, not unknown in the literature (see \cite[p. 12, Problem 1.10]{young} and \cite[pp. 53--54]{dangelo}), that this identity is an instance of a fundamental and ubiquitous process in mathematics, that of 
averaging over a group (a.k.a. Haar integration).  From this point of view, \eqref{eq-pol-classical} is really saying that
\begin{equation}\label{eq-pol-classical2}
		\ipr{x,y}= \int_G \norm{x+g y}^2g dg, 
\end{equation}
where $G=\{\pm1, \pm i\}$ is the four element cyclic group of complex fourth roots of unity, and the integral is that of the complex valued function
$g\mapsto \norm{x+g y}^2g$ on the compact group $G$ with respect to the Haar measure of $G$, normalized to a probability measure. Our 
generalization of \eqref{eq-pol-classical} (Theorem~\ref{thm-main} below) will replace the field of complex numbers by an associative algebra $\Aa$
 over $\rl$ equipped with an involution,
the inner product with a general Hermitian form, the group $G$ with a compact multiplicative group of unitary elements of $\Aa$ (where an element of $\Aa$ is unitary if 
the involution maps it to its inverse).

Several authors have considered the natural problem of obtaining  polarization identities for 
Hermitian forms over scalars other than $\rl$ or $\cx$, and the closely related Jordan-von Neumann theorem in this setting (see, e.g., \cite{kurepa,jamisonthesis,sem, vukman,giarruso, ili1,ili2} etc.).
 The polarization process has also been studied by Waterhouse (\cite{waterhouse}) from a purely algebraic perspective for involutive 
algebras over very general commutative rings.  Our results here have interesting intersections with the generalized polarization identities obtained by these authors, but emphasize the role of 
averaging over  a group as a key feature of the polarization process of recovering a Hermitian form from its restriction to the diagonal.  We also restrict ourselves to algebras over the real numbers
$\rl$, first, to take advantage of classical integration theory, and second, because of the application to a generalized
Jordan-von Neumann theorem (see Theorem~\ref{thm-jvn} below). However, it is clear that ideas and results of this paper
make sense over other fields, and it would be interesting to look at such generalizations.

\subsection{Polarizable algebras and polarizing subgroups} 
Let $\Aa=(\Aa,*)$ be a \emph{ real involutive algebra,}
by which we mean a finite dimensional unital associative algebra over $\rl$ equipped with 
an anti-automorphism $\Aa\to\Aa, \alpha \mapsto \alpha^*$ of order two (the \emph{involution}), i.e., the involution is $\rl$-linear,  $(\alpha\beta)^*=\beta^*\alpha^*$ 
and $(\alpha^*)^*=\alpha$ for $\alpha, \beta\in \Aa$. The multiplicative identity of $\Aa$ will be denoted by $1_\Aa$ (or simply 1 if confusion is unlikely).
We also call real involutive algebras \emph{$*$-algebras} for short, and these $*$-algebras will replace the field $\cx$ in the generalized polarization identity.

Thought of as a finite dimensional real vector space, a  $*$-algebra 
$\Aa$ has a natural linear topology in which the algebraic operations are continuous. We will   always endow $\Aa$ with this topology, and its subsets with the induced subspace topology.  Then the group $\Aa^\times$  of units of $\Aa$ is a locally compact 
topological group. If $G$ is a compact subgroup of $\Aa^\times$, then $G$ is a compact topological group, and therefore admits a (simultaneously left and right invariant) Haar measure,  which will be always normalized to  a probability measure. Given a 
function $f:G\to V$ taking values in a finite dimensional real vector space $V$, we can define 
the integral with respect to Haar measure of $f$, which we will denote by 
\begin{equation}\label{eq-vectint}
\int_G fdg  \quad \text{ or} \quad \int_G f(g)dg
\end{equation}
and can be defined  invariantly as the unique vector in the space $V$ satisfying the condition
\begin{equation}\label{eq-intdefn}
	\phi\left( \int_G fdg\right)= \int_G \left(\phi\circ f\right) \,dg 
\end{equation}
for each linear functional $\phi: V\to\rl$, the integral on the right being the classical integral of a real valued function on $G$ with respect to Haar measure. The existence and uniqueness of the integral in \eqref{eq-vectint} is easily established by choosing a basis of the vector space $V$ and 
working in the corresponding coordinates.

	An element $u\in \Aa$ will be said to be  \emph{unitary} if we have
	$ uu^*=1_\Aa$, and the collection of all unitary elements
	\begin{equation}
		\label{eq-unitary}
	\Gamma(\Aa,*)=\{u\in \Aa: uu^*=1_\Aa\},
	\end{equation}
will be  called the \emph{unitary group } of the algebra $(\Aa,*)$.  It is therefore 
a locally compact topological group, and 
 generalizes 
	the circle group $\{\abs{z}=1\}$ in $\cx$.
	
We will say that a closed subgroup $G\subset \Gamma(\Aa, *)$ \emph{generates} the $*$-algebra $\Aa$ if $\Aa$ is the smallest $*$-subalgebra  of $\Aa$ containing the group $G$. 
It is clear that $\Gamma(\cx) = \{\abs{z}=1\}$ generates $\cx$. Notice that since $G$ is closed under the operations of multiplication and involution of $\Aa$, it follows that $G$ generates $\Aa$ if and only if
\[ \Span_\rl G=\Aa\]
i.e. $G$ generates $\Aa$ also as a vector space. We now make the following definition:
\begin{defn}\label{defn-pol-algebra}
		Let $(\Aa,*)$ be a $*$-algebra. We say that $(\Aa,*)$ is \emph{polarizable} if the unitary group $\Gamma(\Aa,*)$ is compact and generates $\Aa$.
\end{defn}
Polarizable algebras will replace the complex numbers in the general polarization identity \eqref{eq-polarization} below. A key property of polarizable algebras is that a version of Maschke's theorem
holds for them, showing that they are semisimple and leading to a complete classification of these algebras under an appropriate notion of equivalence (see Theorem~\ref{thm-kappa} below).  We now
introduce the groups which will replace, in the general polarization formula,  the group $G$ of   $4$-th roots of unity that appears in \eqref{eq-pol-classical} and \eqref{eq-pol-classical2}:
	\begin{defn}
		\label{defn-pol-group}
		Let $(\Aa,*)$ be a $*$-algebra. A compact subgroup $G \subseteq \Gamma(\Aa,*)$ is said to be  \emph{polarizing} if $G$ generates $\Aa$ as an algebra and 
		\begin{equation}
			\label{eq-gdg}
			\int_G gdg = 0,
		\end{equation}
	where the integral is that of the inclusion function $g:G\to \Aa$  of the group $G$ in the vector space $\Aa$, and the integral is taken with respect to the (bi-invariant) Haar measure.
	\end{defn}
	The following  facts shed some light on the notions introduced in Definitions~\ref{defn-pol-algebra} and  \ref{defn-pol-group}:		\begin{enumerate}[wide, label =(\alph*)]
	\item In a polarizable algebra $(\Aa,*)$, the unitary group $\Gamma(\Aa,*)$ is polarizable. The condition \eqref{eq-gdg} is satisfied for $\Gamma(\Aa,*)$ since $-1_\Aa$, the negative of the identity of the algebra $\Aa$, automatically belongs to $\Gamma(\Aa,*)$, so we have by Haar invariance:
	\begin{equation}\label{eq-gammapol}
	 \int_{\Gamma(\Aa,*)}g dg = \int_{\Gamma(\Aa,*)}(-1_\Aa)g dg = - \int_{\Gamma(\Aa,*)}g dg.
	\end{equation}
\item We will see below that \eqref{eq-gdg} is satisfied for many groups $G$  (Proposition~\ref{prop-firstmoment}).
\item On the other hand, as soon as a $*$-algebra has a polarizing subgroup (or even a compact subgroup of the unitary group that generates the algebra), it is polarizable (see Proposition 2.4 below).

	\item Finally, each polarizable algebra admits a \emph{finite} polarizing subgroup (see Proposition~\ref{prop-finite}) below. It follows
	that the polarization process is always possible in terms of algebraic formulas such as \eqref{eq-pol-classical} rather than integral formulas such as \eqref{eq-pol-int}, where the integral is over a Lie group of positive dimension. 

	\end{enumerate}

\subsection{Generalized Polarization identity} Let $(\Aa, *) $ be a $*$-algebra and let 
$\X$ be a left $\Aa$-module. By an ($\Aa$-valued) \emph{Hermitian form} on 
$\X$ we mean a map $Q:\X\times \X\to \Aa$  such that 
for $x,y,z\in \X$, $\alpha\in \Aa$ we have
\[ Q(x+y,z)=Q(x,z)+Q(y,z), \quad Q(\alpha x,y)= \alpha Q(x,y), \quad \text{ and }\quad Q(x,y)^*=Q(y,x). \]
Then we have the following:
	\begin{thm}
	\label{thm-main} Let $(\Aa,*)$ be a polarizable $*$-algebra. 
	Then there is an element  $\kappa\in \Aa$ such that whenever $\X$ is a left $\Aa$-module and 
	$Q:\X\times \X\to \Aa$ is a Hermitian form, then we have the polarization identity
	\begin{equation}\label{eq-polarization}
	Q(x,y)= \kappa\int_G q(x+gy)g dg,  \quad  x,y \in \X	\end{equation}
where  $q(x)=Q(x,x)$, $G\subset \Gamma(\Aa)$ is a polarizing subgroup, 
 and the integral is that of an $\Aa$-valued function on $G$ with respect to the normalized Haar measure.
\end{thm}

		We emphasize  that in Theorem~\ref{thm-main}, the constant  $\kappa$ depends only on the (polarizable) $*$-algebra $(\Aa,*)$
		 and not on the  particular polarizing group $G$.  This element $\kappa=\kappa(\Aa,*)$ will be called the \emph{polarization constant} of $(\Aa,*)$, and in Theorem~\ref{thm-kappa} below we show how to compute it.

\subsection{More definitions and constructions related to $*$-algebras}
Before we proceed further, we collect a few definitions. 
 
\begin{enumerate}[wide]
	\item If $(\Aa, \sharp)$ is a $*$-algebra, there is a natural way to define  an induced involution   on the matrix algebra $M_n(\Aa)$ by setting
	\begin{equation}\label{eq-induced}
		A\mapsto (A^\sharp)^T,
	\end{equation}
	where $A^\sharp$ is the matrix obtained from $A$ by applying the involution $\sharp$ to each entry, and $B^T$ is the transpose of the matrix $B$.   We call this  involution on $M_n(\Aa)$ the \emph{involution induced by $\sharp$.}
	\item Each of the real division algebras $\D=\rl,\cx, \hx$ has its own standard  involution  traditionally called a ``conjugation", which is the identity for $\rl$, the complex conjugation 
	for $\cx$ and the quaternionic conjugation
	\begin{equation}\label{eq-conjugation}
		 a+ib+jc+kd\mapsto a-ib-jc-kd
	\end{equation}
for $\hx$.  We denote the conjugation operation in each case by $z\mapsto \ol{z}$, following the standard convention. 
The involution induced by the conjugation map on $\D$ on the algebra $M_n(\D)$ will be called the ``conjugate-transpose" involution of 
$M_n(\D)$. We will see later that $M_n(\D)$ has many other involutions apart from the conjugate-transpose (see Proposition~\ref{prop-involution} below).
\item In the category of involutive algebras, a morphism is called a \emph{$*$-homomorphism}:
\[ \phi:  (\Aa, *)\to(\mathfrak{B},\sharp)\]
where $\phi$ is a homomorphism of unital associative algebras, and preserves the involution, i.e.
\[ \phi(\alpha^*)= \phi(\alpha)^\sharp.\]
It is clear what is meant by a $*$-isomorphism, or a $*$-subalgebra. 
\item Given a finite collection  $\{ (\Aa_j, *_j), 1\leq j \leq N\}$ of $*$-algebras, we define its \emph{$*$-direct sum} be the direct sum 
$\bigoplus_{j=1}^N \Aa_j$ equipped with the ``direct sum involution"
\[ (\alpha_1, \dots, \alpha_N)^*=(\alpha_1^{*_1},\dots, \alpha_N^{*_N}), \quad \alpha_j\in \Aa_j, 1\leq j \leq N.\]
\end{enumerate}

\subsection{Polarization constants}
We now determine the structure of polarizable algebras and use it to compute the polarization constant of an algebra:
	\begin{thm}
		\label{thm-kappa}
			Let $(\Aa,*)$ be a $*$-algebra.  Then $(\Aa,*)$ is polarizable if and only if  there is a $*$-isomorphism 
		\begin{equation}\label{eq-isom}
		\phi:  \bigoplus_{j=1}^N (M_{n_j}(\D_j),*)\to (\Aa,*),\	\end{equation}
	where in the left is the $*$-direct sum of a   a finite number of matrix algebras over real division rings, 
		 each endowed with the ``conjugate-transpose" involution.  The polarization constant of $(\Aa,*)$ is 
			 \begin{equation}\label{eq-kappacomp}
		 	 \phi\left( \frac{n_1\delta_1}{(n_1-1)\delta_1 +2 }I_{n_1}, \dots,\frac{n_j\delta_j}{(n_j-1)\delta_j +2 }I_{n_j},\dots, \frac{n_N\delta_N}{(n_N-1)\delta_N+2 }I_{n_N}\right), \end{equation}
		where $\delta_j=\dim_{\rl}(\D_j)=1,2$ or $4$ denotes the dimension of the division algebra $\D_j$ as a real vector space and $I_{n_j}$ is the $n_j\times n_j$ identity matrix  of $M_n(\D_j)$ for $1\leq j \leq N$.
	\end{thm}
The fact that a polarizable algebra such as $(\Aa,*)$ has the above structure (i.e. it is semisimple) allows us to understand the module $\X$ and Hermitian form $Q$ of Theorem~\ref{thm-main}
in terms of $\Aa$ (see \cite{lam}).  Since the precise structure of the module $\X$  does not play any role in the polarization phenomenon of Theorem~\ref{thm-main}, we do not discuss this further.

\subsection{Generalized Jordan-von Neumann-Jamison theorem}
The Jordan-von Neumann theorem (\cite{jvn})
solves the problem of 
geometrically characterizing inner product spaces among  normed linear spaces  over $\cx$ or $\rl$: the norm of 
a normed space $(X,\norm{\cdot})$ arises from an inner product (i.e., there is an inner product $\ipr{\cdot,\cdot}$ on $X$ such that $\norm{x}= \sqrt{\ipr{x,x}}$ for each $x\in X$)
if and only if the \emph{parallelogram identity}
\begin{equation}\label{eq-classic-parallelogram}
	\norm{x+y}^2+\norm{x-y}^2= 2\left(\norm{x}^2+\norm{y}^2\right)
\end{equation}
holds for all $x,y\in X$.  In \cite{jamisonthesis}, the field of scalars was already extended to the 
quaternions (and octonions).  Here, we will prove a generalization to Hermitian forms over 
polarizable algebras. 

The analog of the map $x\mapsto \norm{x}^2$  occurring in  formulas such as \eqref{eq-pol-classical} and \eqref{eq-classic-parallelogram} 
in the  situation of non-commutative scalars from $\Aa$ will be called a \emph{quadrance}, i.e., a ``squaring",
a name inspired by the unorthodox work of Norman J. Wildberger \cite{divine}. 

\begin{defn}
	\label{defn-quadrance}  If $(\Aa,*)$ is a $*$-algebra,  and $\X$ a left $\Aa$-module, 
	 a map $q \from \X \to \Aa$ is called a \emph{quadrance} on $\X$ if 
	\begin{enumerate}
		\item $q(x)=q(x)^*$ for all $x \in \X$
		\item $q(\alpha x) = \alpha q(x) \alpha^*$ for all $x\in \X,\alpha \in \Aa$
		\item \label{item-topological} For  $x,y\in \X$ the map
		$ \rl \times \Aa\to \Aa$
		given by
		\begin{equation}\label{eq-continuity}
			(\lambda, \alpha) \mapsto q(\lambda x+ \alpha y)
		\end{equation}
		is continuous.
	\end{enumerate}
\end{defn}

Notice that the square of a norm on  a real or complex  vector space is obviously a quadrance. In this case, the vector space acquires a norm-topology, in which 
the norm is continuous, by the ``reverse triangle inequality" $\abs{\norm{x}-\norm{y}} \leq \norm{x-y}$, which is an immediate consequence of the ``triangle inequality"   $\norm{x+y}\leq \norm{x}+\norm{y}$  postulated 
on norms. Notice also that we have not endowed the module $\X$ with a topology, but we impose 
	the  continuity condition \eqref{eq-continuity}  on a quadrance $q$. 
	We show below in Proposition~\ref{prop-diag-rest} that the restriction of a Hermitian form to the diagonal is a quadrance.
	Conversely we have the following generalization of the Jordan-von Neumann theorem to non-commutative scalars:

	\begin{thm}\label{thm-jvn}	Let $(\Aa,*)$ be a polarizable $*$-algebra, let $\X$ be a left $\Aa$-module,
	and let $q$ be a quadrance on $\X$. 
	 Then there is a  Hermitian form $Q$ on $\X$ such that $q(x)=Q(x,x)$
 if and only if $q$  satisfies the \emph{classical parallelogram 
	identity }
	\begin{equation}
		\label{eq-classpar}
		q(x+y)+q(x-y)=2(q(x)+q(y)),
	\end{equation}
	 for all $x,y\in \X$. If such a form $Q$ exists, then it is unique. 
 
\end{thm}
\subsection{Examples} In the last section of the paper, we show that group algebras of finite groups over the reals, as well as 
real Clifford algebras are polarizable, and determine some polarizing subgroups, as well as the polarization constants. This gives examples of concrete polarization identities in new contexts. The case of classical Clifford numbers $\mathsf{C}_{0,n}$ associated 
to negative definite inner 
products was considered in \cite{giarruso}. We also make a few comments about polarizing subgroups in matrix algebras. 

\subsection*{Acknowledgments:} We thank the referee for careful reading of the paper, many helpful comments and suggestions, and pointing to many references unknown to us, all of which led to improvements of the paper.  We would also like to thank Preeti Raman, Nicolas Young, S. Viswanath, Amritanshu Prasad and John D'Angelo
for their comments and suggestions. Substantial parts of 
this paper are based on the MA thesis of the first-named author under the supervision of the second-named author. The first-named author would like to thank his other committee members, Jordan Watts and Lisa DeMeyer for their suggestions. 

\section{Structure of polarizable algebras}
In this section our goal is to prove the first half  of Theorem~\ref{thm-kappa}, i.e.,
a $*$-algebra is polarizable if and only if it is isomorphic to  a  $*$-direct sum of matrix algebras over real division rings, each endowed with the conjugate-transpose involution.

\subsection{Involutions on matrix algebras and compact unitary groups}  The following result is needed to 
prove the existence of isomorphism \eqref{eq-isom}:
\begin{prop}\label{prop-noncompact}
	 Let $\D=\rl,  \cx$ or $\hx$ be a real division algebra, and let $\sharp$ be an involution on $M_n(\D)$. If the involutive algebra $(M_n(\D), \sharp)$ 
	is polarizable then it is $*$-isomorphic to $(M_n(\D),*)$ where $*$ is the conjugate-transpose involution of $M_n(\D)$. 
	
\end{prop}

The proof will employ brute force, using a  classification of  all the involutions of the matrix algebras $M_n(\D)$, and a characterization of their unitary groups, and 
then singling out the ones which have polarizable unitary groups.  Since we have not been able to locate a full statement of this classification in 
the literature, we include it below in Proposition~\ref{prop-involution}, and sketch a proof based on the general results in \cite{scharlau} 
on involutions of matrix algebras.

 We begin by 
recalling some facts and introducing some notation to simplify the presentation of this result.
\begin{enumerate}[wide]
	\item On the quaternions $\hx$, it is known that any involution is either the conjugation \eqref{eq-conjugation}, or is a so-called \emph{nonstandard involution} (see \cite[Theorem~2.4.4]{rodman}):  an example of a nonstandard involution is the mapping 
	 $x\mapsto \wt{x}$ given for $x=x_0+x_1i+x_2j+x_3k$ by
		\begin{equation}\label{eq-tilde}
\wt{x}=x_0-x_1i+x_2j+x_3k,
	\end{equation}
	and any other nonstandard involution is conjugate to it. A computation shows that  the nonstandard involution  \eqref{eq-tilde} and 
	the conjugation $x\mapsto \ol{x}$ are related by 
	\begin{equation}\label{eq-bartilde}
	\wt{x}= i\ol{x}i^{-1}= - i \ol{x}i.
	\end{equation}
	Given a matrix $A$ of 
	quaternions, we denote by $\wt{A}$ the matrix obtained by applying the involution $\wt{\cdot}$ to each element of $A$. 
	\item Let $J_1 \in M_2(\rl)$ be the matrix 
	$\displaystyle{ J_1 = \begin{pmatrix} 0 & -1 \\ 1 & 0 \end{pmatrix} }$ 
	and for each positive integer $m$ let $J_m$ be the $2m\times 2m$ block diagonal matrix with $m$ blocks each equal to $J_1$:
	\[J_m = \diag( J_1 , \dots , J_1 ). \] 
	
	\item Let $p,q$ be nonnegative integers such that $p+q=n$. We
	denote by $I_{p,q}$ the diagonal matrix where the first $p$ diagonal entries are  +1s 
	and the remaining  $q$ entries are -1s:
	\[ I_{p,q}=\diag\left({1,\dots, 1},{-1,\dots, -1}\right).\]

\end{enumerate}

\begin{prop} \label{prop-involution} The  rows of the following table give a complete list of nonisomorphic involutions on the algebras $M_n(\D), \D=\rl,\cx, \hx$, and characterize the corresponding unitary groups. More precisely,   for $\D=\rl, \cx$ or $\hx$, if 
	$\sharp$ is an involution on the 
real associative algebra  $M_n(\D)$, then $(M_n(\D),\sharp)$ is $*$-isomorphic to $(M_n(\D), \flat)$, where $\flat$ is an involution in 
some row of the table in the column ``Involution"  associated with the algebra $M_n(\D)$ in the first column. 
\begin{center}

\renewcommand{\arraystretch}{1.5} 
\emph{
\begin{tabular}{|c|c|c|l|l|}
		\hline
{	Algebra}	& {No}. &{Parameters}&  {Involution} &{Unitary Group} \\
		\hline\hline
	\multirow{2}{*}{$M_n(\rl)$}&$1_{p,q}$&$p+q=n, p\geq q$ &$A\mapsto I_{p,q}A^TI_{p,q}$&$O(p,q)=\{A^TI_{p,q}A=I_{p,q}\}$ \\
	\cline{2-5} &2&$n=2m$ even& $A\mapsto - J_m A^T J_m$ &$Sp(2m, \rl)=\{A^TJ_{m}A=J_{m}\}.$\\\hline
\multirow{3}{*}{$M_n(\cx)$}	&$3_{p,q}$& $p+q=n, p\geq q$&$A\mapsto I_{p,q}A^*I_{p,q}$ & $U(p,q)=\{A^*I_{p,q}A=I_{p,q}\}$ \\
\cline{2-5} &4&&$A\mapsto A^T$&$O(n,\cx)= \{A^TA=I\}$\\
\cline{2-5} &5&$n=2m$ even&$A\mapsto- J_m A^T J_m$ &$Sp(2m, \cx)=\{A^TJ_{m}A=J_{m}\}.$\\
		\hline
\multirow{2}{*}{$M_n(\hx)$}	&$6_{p,q}$& $p+q=n, p\geq q$&$A\mapsto I_{p,q}A^*I_{p,q}$ & $Sp(p,q)=\{A^*I_{p,q}A=I_{p,q}\}$ \\  
\cline{2-5} &7&&$A\mapsto\wt{A}^T$& ${Sp_{\mathrm{NS}}(n)=}\{\wt{A}^TA=I\}$\\
		\hline
\end{tabular}}\end{center}
The rows numbered $1_{p,q}, 3_{p,q}$  and $6_{p,q}$ each  stand for multiple rows (and therefore multiple non $*$-isomorphic $*$-algebras), with a separate row for each ``signature" $(p,q)$.

\end{prop}
If $R$ is an involution of $M_n(\D)$, then the restriction of $R$ to the center $\mathbb{K}$ of $M_n(\D)$ is a field automorphism of 
$\mathbb{K}$ fixing $\rl$, which is either the identity, or is of order 2. 
When $\mathbb{\D}=\rl$ or $\hx$, then $\mathbb{K}$ can be identified with $\rl$, so all involutions restrict to the identity. 
On the other hand, if $\D=\cx$, the center $\mathbb{K}$ consists of the complex scalar matrices, so an
	 involution can be either restrict to the identity (a so called involution of the \emph{first kind}),
	 or its restriction to the center $\mathbb{K}\cong \cx$
can be the complex conjugation map (an involution is of the \emph{second kind}).

The following is the special case for 
the base field $\rl$ of 
a well-known result on the classification of involutions of central simple algebras over a field  (see \cite[Chapter 8, Theorem 7.4]{scharlau}). 
The proof is straightforward using the Skolem-Noether characterization of automorphisms of simple algebras, and we refer the reader 
to \cite{scharlau} for details:
\begin{lem}
\label{lem-scharlau}	
Let  $\D=\rl,\cx$ or $\hx$ and $n$ a positive integer.
	\begin{enumerate}
		\item Every involution $R$ on $M_n(\D)$ is of the form
		\begin{equation}
			\label{eq-SU}
			A^R = U A^S U^{-1} , \quad \text{for all  } A\in M_n(\D)
		\end{equation}	 
		 where
		 \begin{enumerate}
		 	\item $S$ is the
		 	  transposition map on $M_n(\D)$ if  $\D=\rl$ or  $\D=\cx$ and $R$ is of the first kind, and the
		 conjugate-transpose map  otherwise, and 
		 	\item $U$ is invertible in $M_n(\D)$ and satisfies $U^S=\pm U$.
		 \end{enumerate}
	 
		\item If $\sharp, \flat$ are involutions on $M_n(\D)$, then  $(M_n(\D),\sharp)$ and $(M_n(\D),\flat)$ are $*$-isomorphic if and only if  the following  conditions hold:
		
		\begin{enumerate}
		\item If $\D=\cx$ then both $\flat$ and $\sharp$ are of the first kind or of the second kind.
		\item If 
			\[ 	A^\sharp = U A^S U^{-1}, \quad 	A^\flat = V A^S V^{-1}\]
			are the representations of the two involutions given by \eqref{eq-SU} (with $S$ as in 
			part 1(a)), then we have
			\begin{equation}
				\label{eq-multcong}
				V = \lambda W U W^S
			\end{equation}
			with  $\lambda\not=0$ in the center of $\D$ and $W$ inveritible in $M_n(\D)$.
		\end{enumerate}
	\end{enumerate}
\end{lem}

\begin{proof}[Proof of Proposition~\ref{prop-involution}]
	In view of  Lemma~\ref{lem-scharlau}, the proof  is reduced to the determination of equivalence classes in $M_n(\D)$ under the 
	equivalence relation \eqref{eq-multcong} (so-called \emph{multiplicative congruence}).  For $\D= \rl, \cx$, wherever possible, we 
	 will  use classical results in linear algebra  which 
	classify matrices up to \emph{congruence} (i.e., with $\lambda=1$ in \eqref{eq-multcong}). 
	In the case of the quaternions, we will need to use facts about spectral theory over $\hx$. 
	Once involution is reduced to a normal form, the computation of the unitary group of the last column is straightforward.

\textbf{	The case $\D=\rl$.} Thanks to Lemma~\ref{lem-scharlau} each involution $M_n(\rl)$ is  of the form $A\mapsto UA^T U^{-1}$ where $A\mapsto A^T$ is the transposition operation, and $U^T=\pm U$, i.e. $U$ is either symmetric or anti-symmetric. We consider the two cases separately.

\begin{enumerate}[wide]
	\item \emph{$U$ symmetric.} The eigenvalues of the symmetric and invertible matrix $U$ are real and nonzero, and suppose that $U$ has $p$ positive and $q=n-p$ negative eigenvalues. 
	 By the real form of  Sylvester's Law of of Inertia (\cite[Theorem~6.11.1]{herstein}), the invertible symmetric matrix $U$ is congruent to the matrix $I_{p,q}$, i.e., there is a matrix $W$ such that $I_{p,q}=WUW^T$. If $p\geq q$, this
	 shows that $I_{p,q}$ and $U$ are multiplicatively congruent. If $q>p$ then the relation $I_{q,p}= (-1) WUW^T$ holds, showing that $U$ is multiplicatively congruent to $I_{q,p}$. Case $1_{p,q}$ of the table follows, on noting that $I_{p,q}^{-1}=I_{p,q}$.
	\item \emph{$U$ is anti-symmetric.} Using the real spectral theorem, it is not difficult  to show that  (see \cite[Corollary~2.5.11]{horn})
	$n=2m$ is even and 	there is an orthogonal matrix $V\in O(n)$ such that $U=V\diag(r_1J_1,\dots, r_mJ_1)V^T$, where each $r_j>0$.	Taking $W=V\diag (\sqrt{r_1}, \sqrt{r_1}, \dots, \sqrt{r_m}, \sqrt{r_m})$, we see that $U=WJ_m W^T$, so $U$ is (multiplicatively) congruent to $J_m$, and consequently by Lemma~\ref{lem-scharlau}, the involution $\sharp$ is equivalent to the involution $A\mapsto J_m A^T J_m^{-1}=-J_m A^T J_m$ of  row 2, since $J_m^{-1}=-J_m$. 
	\end{enumerate}

\textbf{The case $\D=\cx$.} Given an involution on $M_n(\cx)$, it can be either of the first or second kind. We consider these two cases separately.

First suppose that we are given an involution of the first kind on $M_n(\cx)$, i.e., it acts as the identity on the complex scalar matrices that form the center of $M_n(\cx)$. Then by Lemma~\ref{lem-scharlau},  such an involution is of the form $A\mapsto UA^TU^{-1}$ where $A\mapsto A^T$ is the transposition involution on complex matrices and 
$U^T=\pm U$. Again we consider the two cases:
\begin{enumerate}[wide]
\item \emph{$U$ symmetric, i.e. $U^T=U$.} By a classical result of linear algebra (see \cite[Theorem~4.5.12]{horn})  for the nonsingular 
complex symmetric  matrix $U$, there is an invertible $W$ such that $ WUW^T=I$,  which shows that $U$ is (multiplicatively) congruent to the identity matrix $I$. The case of row 4 follows.
\item \emph{$U$ is anti-symmetric, i.e. $U^T=-U$.} By an application of the spectral theorem (see \cite[Theorem~7, p. 481]{hua})
we can show that there is a $V$ such that we have \[U=V\diag(r_1J_1,\dots, r_mJ_1)V^T\] where $n=2m$ (and hence even) and each $r_j>0$. 
Now we take \[W=V\diag (\sqrt{r_1}, \sqrt{r_1}, \dots, \sqrt{r_m}, \sqrt{r_m}),\] and see that $U=WJ_m W^T$, so $U$ is (multiplicatively) congruent to $J_m$, and consequently by Lemma~\ref{lem-scharlau}, we get row 5.

\end{enumerate}

Now suppose that we are given an involution of the second kind on $M_n(\cx)$, so that by Lemma~\ref{lem-scharlau}, we see that such an involution is of the form $A\mapsto UA^*U^{-1}$, where $A\mapsto A^*$ is the standard conjugate-transpose involution of the algebra $M_n(\cx)$, and $U^*=\pm U$. We consider the two cases:
\begin{enumerate}[wide]
	\item \emph{$U$ is Hermitian, i.e. $U^*=U.$} This case is similar to the case of row number $1_{p,q}$, using the Hermitian version of Sylvester's Law of Inertia (\cite[Theorem~4.5.8]{horn}) , which shows that there is a $W\in GL_n(\cx)$ such that $I_{p,q}=WUW^*$, where $p$ and $q$ are the numbers of positive and negative eigenvalues of the matrix $U$. If $p<q$ we again write $(-1)WUW^*=I_{q,p}$. This gives row $3_{p,q}$ of the table. 
	
	\item \emph{$U$ is Skew-Hermitian, i.e. $U^*=-U$} Since $(iU)^*=iU$, by the arguments of the preceding case there is a  signature $(p,q)$ and a matrix $W$ such that
	\[ I_{p,q}=W(iU)W^*= iWUW^*\]
	If $p<q$ we again have $I_{q.p}= -iWUW^*$, so we have that $U$ is multiplicatively congruent to one of the matrices $I_{p,q}$ with $p\geq q$, and we are again in row $3_{p,q}$.
\end{enumerate}

\textbf{The case $\D=\hx$.}   We will use the following quaternionic version of the spectral theorem, where the quaternions of the form 
$a+ib, a,b \in \rl$ (i.e. quaternions with vanishing $j$ and $k$ components) are identified with the complex numbers $\cx$:

\textbf{Result:} (see \cite{spectral}, and  cf. \cite[Theorem~4.1.12]{rodman})\emph {  Suppose $U \in M_n(\hx)$ is a normal matrix, i.e. $U^* U = UU^*$. Then there exists a $V \in Sp(n)$ and a diagonal matrix $D$ with entries in the closed upper half plane in $\cx\subset \hx$ such that $U=VDV^*$. }

By Lemma~\ref{lem-scharlau} every involution on $M_n(\hx)$ is of the form $A\mapsto UA^*U^{-1}$ where $U \in GL_n(\hx)$ is such that $U^* = \pm U$. In both cases, $U$ is normal so that we can apply the quaternionic spectral theorem as stated above.

\begin{enumerate}[wide]
	\item\emph{$U$ is quaternionic Hermitian, i.e. $U^*=U$:} Then there exists a $V \in Sp(n)$ such that $V^* U V = D$, a diagonal with values in the upper half plane. Further $D^* = V^* U^*V = V^* U V = D$ and it follows that $D$ is a real diagonal matrix. Applying the argument used to deduce
Sylvester's law of inertia from the spectral theorem (see the proof of \cite[Theorem~6.11.1]{herstein}), we see that there are $p,q\geq 0, p+q=n$
and $W \in GL_n(\hx)$ such that $I_{p,q}=W U W^*$. Again if $q>p$, we can write $I_{q,p}=(-1)W U W^*$, showing that $U$ is multiplicatively 
congruent to $I_{p,q}$ for some $p\geq q$.
	
	\item \emph{$U$ is quaternionic skew-Hermitian, i.e. $U^*=-U$} 
If $U^* = -U$, then there exists a $V \in Sp(n)$ such that $V^* U V = D$, a diagonal matrix with values in the upper half plane. In this case, however, $D^* = V^* U^* V = -V^* U V = -D$ and it follows that $D = \diag(ir_1,\dots, ir_n)$ where each $r_j>0$ is a positive real number.
 Setting 
\[W = V\diag\left(\sqrt{r_1}, \dots, \sqrt{r_n}\right) \] yields $U = W (iI) W^*,$  so that  the algebra $M_n(\hx)$ with this involution is $*$-isomorphic 
to the algebra $(M_n(\hx), \sharp)$, where $A^\sharp= (iI)A^*(iI)^{-1}=iA^*i^{-1}$. Denoting the entry in the $\lambda$-th row and $\mu$-th 
column of the matrix $A$ by $(a_{\lambda\mu})$, we see using \eqref{eq-bartilde} that the $(\lambda, \mu)$-th entry of $iA^*i^{-1}$ is 
\[ i\cdot\ol{a_{\mu,\lambda}}\cdot i^{-1}= \wt{a}_{\mu,\lambda},\]
where the tilde denotes the nonstandard involution of   \eqref{eq-tilde}. Therefore, we get the involution $A\mapsto \wt{A}^T$ of 
row (7) of the table (cf. \cite[section~3.6]{rodman}) and it is clear that the corresponding unitary group is $\{\wt{A}^TA=I\}$, which
does not seem to have a standard name in the literature, but perhaps may be called the nonstandard symplectic group $Sp_{\mathrm{NS}}(n)$.
\end{enumerate}
 \end{proof}

\begin{proof}[Proof of Proposition~\ref{prop-noncompact}]  We will show that if $\sharp$ is not the conjugate-transpose involution on $M_n(\D)$, then the algebra $(M_n(\D),\sharp)$ is not polarizable, i.e., the unitary groups in the table of Proposition~\ref{prop-involution}
	except those in rows $1_{n,0}, 3_{n,0}$ and $6_{n,0}$ are not compact  or do not generate
	$M_n(\D)$.

	When $n=1$, and $\D=\rl$, then the only possibility is row $1_{1,0}$ so there is nothing to show. For $\D=\cx$ the only possibility (except $3_{1,0}$) is row 4, where the unitary group is \[\Gamma(M_1(\cx), \mathrm{id})=\{z\in M_1(\cx)=\cx: z^2=1\}=\{\pm 1\}.\]
	 The linear span of this group is $\rl$ and therefore it is not polarizing.  For $\D=\hx$, the  only possibility we need to consider is 7, where the unitary group is
	 \[\{a\in \hx: \wt{a}a=1_\hx\}.\]
	  It is easy to see that a quaternion satisfying this system of four real equations is of the form $a=\cos\theta+i \sin \theta$ for some $\theta\in \rl$. Therefore, the unitary group in this case generates a subalgebra of $\hx$ isomorphic to $\cx$, and not the whole of $\hx$. The result is proved when $n=1$.

	Now  let $n\geq 2$. 
	In the numbering system of the table in Proposition~\ref{prop-involution},  the conjugate transpose involutions correspond to rows $1_{n,0}$ (for
$\rl$), $2_{n,0}$ (for $\cx$) and $6_{n,0}$ (for $\hx$).
We will show that all the other groups in the table are noncompact, thus showing that they are nonpolarizable. If $n=2$ and $p=q=1$, for number $1_{1,1}$ it is well-known that the group $O(1,1)$ is not compact, which can be seen for example by looking at the sequence of matrices
\[ \dfrac{1}{2}\begin{pmatrix}
	n+\dfrac{1}{n}& 	n-\dfrac{1}{n}\\
	& \\
	n-\dfrac{1}{n}&	n+\dfrac{1}{n}
\end{pmatrix}\]
in $O(1,1)$ which does not have a limit point. Now since $O(1,1)$ can be embedded as a closed subgroup of $O(p,q)$ if $n\geq 3$ it follows that 
$O(p,q)$ is noncompact if $p\geq q\geq 1$.

The inclusions $O(p,q)\subset U(p,q)\subset Sp(p,q)$ (obtained by extending scalars) show that rows $3_{p,q}$ and $6_{p,q}$ also have noncompact
unitary groups if $p,q\geq 1$.

In row no. 4, the group $O(n,\cx)$ is a complex affine variety in $\cx^{n^2}=M_n(\cx)$ of dimension equal to that of the Lie algebra $\{A+A^T=0\}$ which is therefore $\frac{1}{2}n(n-1)$. If $n\geq 2$, this is at least 1, so $O(n,\cx)$ is  noncompact.

The map $f:\cx\to \hx$ given by $f(x+iy)=x+jy$ is an $\rl$-algebra monomorphism, and satisfies $f(x)=\wt{f(x)}$ for each $x$, where the tilde has 
the same meaning as in \eqref{eq-tilde}. We may therefore 
define an algebra homomomorphism  $f:M_n(\cx)\to M_n(\hx)$ by applying it elementwise. This maps the group $O(n,\cx)$ into a closed subgroup of the group $Sp_{\mathrm{NS}}(n)=\{\wt{A}^TA=I\}$, which is therefore noncompact.

Using the formula for the inverse of a $2\times 2$ matrix, we see that a matrix $A\in M_2(\rl)$ satisfies the condition $A^TJ_1 A=J_1$ if and only if  $\det A=1$, so $Sp(2,\rl)=SL(2,\rl)$. The noncompactness of $SL(2,\rl)$ can be seen, e.g., by looking at the sequence of matrices
$\left\{ \begin{pmatrix}
	n&1 \\n-1&1
\end{pmatrix}\right\}
$. Using the obvious inclusion of $Sp(2,\rl)$ in $Sp(2m, \rl)$ we see that the latter is noncompact for each $m\geq1$. Similarly, the inclusion of $Sp(2m,\rl)$ in $Sp(2m, \cx)$ shows that the latter is noncompact. This completes the proof.
\end{proof}

\subsection{Semisimplicity of polarizable algebras} We begin with the following simple application of the Weyl averaging trick:

		\begin{prop}\label{prop-structure1}
	Let $(\Aa,*)$ be an involutive algebra and suppose there is  a compact subgroup $G$ of $\Gamma(\Aa,*)$ which generates $\Aa$. Then
	\begin{enumerate}
	 \item there exists an inner product on the $\rl$-vector space $\Aa$:  \[\ipr{\cdot,\cdot} \from \Aa \times \Aa \to \rl\]
	 such that 
	for  $\alpha,\beta,\eta \in \Aa$ we have
		\begin{equation}
		\label{eq-iprb}
		\ipr{\eta \alpha ,  \beta} = \ipr{\alpha , \eta^* \beta}.
	\end{equation}
\item  the $*$-algebra $(\Aa,*)$ is $*$-isomorphic to a $*$-subalgebra of $(M_n(\rl), T)$, were $n=\dim_\rl \Aa$ (as a vector space) and $T$ denotes the transposition involution.
\item  the $*$-algebra $(\Aa,*)$ is polarizable. 
	\end{enumerate}
\end{prop}

\begin{proof}\begin{enumerate}[wide]
		\item Let $\ipr{\cdot,\cdot}_0 \from \Aa \times \Aa \to \rl$ be a real inner product on $\Aa$ , i.e. a positive-definite, symmetric, bilinear form.  Then
	\[ \ipr{\alpha,\beta} := \int_G \ipr{g \alpha , g \beta  }_0 dg \]
	is also an  inner product, and by Haar invariance we have for this inner product
	for each $g\in G$, $\alpha, \beta\in \Aa$ that
		\begin{equation}
		\label{eq-ipra}
		\ipr{g\alpha , g \beta} = \ipr{\alpha , \beta}.
	\end{equation}  
Therefore
		\[ \ipr{h \alpha, \beta } =  \ipr{h \alpha, h h^* \beta } = \ipr{\alpha, h^* \beta }\]
	for $\alpha, \beta \in \Aa$ and $h \in G$. By the $\rl$-bilinearity of the inner product, and the fact that $\Aa=\mathrm{span}_{\rl}G$  the equality \eqref{eq-iprb}
	follows. 
	\item Let 
	\begin{equation}\label{eq-rho}
\rho:\Aa \to \End_{\rl}(\Aa)	\end{equation}
	be the injective $\rl$-algebra homomorphism given by $\rho(\alpha)(\beta)=\alpha\beta$. 
On the endomorphism algebra $	\End_{\rl}(\Aa)$ there is an involution 
$\sharp$ given by taking adjoints with respect to the inner product of part (1) above, i.e. 
for an $\rl$-linear map $A:\Aa\to \Aa$, the map
$A^\sharp:\Aa\to \Aa$ is characterized by the fact that $\ipr{A\alpha, \beta}= \ipr{\alpha, A^\sharp \beta}$. Notice that \eqref{eq-iprb} says that for each $\eta\in \Aa$, we have $\rho(\eta)^\sharp= \rho(\eta^*)$, i.e., $\rho$ is a $*$-homomorphism
from $(\Aa, *)$ into $(\End_{\rl}(\Aa), \sharp)$, which is known to be injective, so that $(\Aa, *)$ is isomorphic to a $*$-subalgebra of $(\End_{\rl}(\Aa), \sharp)$. Now choosing a basis of the $\rl$-vector space
$\Aa$,  the algebra $(\End_{\rl}(\Aa), \sharp)$ is $*$-isomorphic to $(M_n(\rl), *)$ and the result follows. 

\item Identifying $\End_\rl(\Aa)$ with the matrix algebra $M_n(\rl)$ by choosing a basis, we see that the map $\rho$ of \eqref{eq-rho}
\[ \Gamma(\Aa,*)=\rho^{-1}(O(n)),\]
where $O(n)$ is the orthogonal group. Since $\rho$, being linear and injective, is topologically a proper map, it follows that $\Gamma(\Aa, *)$ is compact. Since the subgroup $G$ of 
$\Gamma(\Aa, *)$ generates $\Aa$, it follows a fortiori that $\Gamma(\Aa, *)$ generates $\Aa$.  The result follows. 
		\end{enumerate}
	\end{proof}

We are now ready to prove the first half of Theorem~\ref{thm-kappa}, which we state as a separate proposition for convenience. It can be thought of as a version of Maschke's theorem.
\begin{prop}\label{prop-kappa1}
		Let $(\Aa,*)$ be a polarizable $*$-algebra.  Then there is a $*$-isomorphism 
	\begin{equation}\label{eq-isom1}
		\phi:  \bigoplus_{j=1}^N (M_{n_j}(\D_j),*)\to (\Aa,*),\	\end{equation}
	where in the left is the $*$-direct sum of a    finite number of matrix algebras over real division rings, 
	each endowed with the ``conjugate-transpose" involution.  
\end{prop}
	\begin{proof}Let $\mathfrak{j}$ be a left ideal in the algebra $\Aa$, and 
		let $\mathfrak{k}$ be the $\rl$-linear subspace of $\Aa$ which is its orthogonal complement with respect to the inner product of Proposition~\ref{prop-structure1}, i.e.
		\[ \mathfrak{k}=\{\alpha\in \Aa: \ipr{\alpha, \beta}=0, \text{ for each } \beta \in \mathfrak{j}\}.\]
		Now, $\mathfrak{k}$ is also a left ideal in the algebra $\Aa$, since by \eqref{eq-iprb}, we have for $\alpha\in \mathfrak{k}$ and $\eta\in\Aa$, 
		\[	\ipr{\eta \alpha ,  \beta} = \ipr{\alpha , \eta^* \beta}=0,\]
		for each $\beta\in \mathfrak{j}$ since $\mathfrak{j}$ is a left ideal. It follows that $\Aa$ is semisimple as a left module over itself, i.e., it is a semisimple $\rl$-algebra. By a famous theorem of Wedderburn (e.g. \cite{lam}), the algebra $\Aa$ is isomorphic to a direct sum of matrix algebras over $\rl$-division algebras. This shows the existence of an isomorphism of $\rl$-algebras:
		\[ 		\phi:  \bigoplus_{j=1}^N M_{n_j}(\D_j)\to \Aa,\]
where for each $j$, $\D_j$ is one of the three real division algebras.

Let $\sharp$ denote the involution on the algebra $\mathfrak{B}=\bigoplus_{j=1}^N M_{n_j}(\D_j)$ obtained by pulling back the involution $*$ of $\Aa$ via $\phi$, so that 
$\phi$ is a $*$-isomorphism of $\Bb$ with $\Aa$: $\phi(x^\sharp)= \phi(x)^*$. We will now show that 
the involution $\sharp$ preserves the direct sum structure of $\Bb$, i.e.,
there exists for each $j$, an involution $\sharp_j$ of $M_{n_j}(\D_j)$ such that 
\[(x_1,\dots, x_N)^\sharp= (x_1^{\sharp_1}, \dots, x_n^{\sharp_n}), \]
i.e. $\Bb$ is the $*$-direct sum of the algebras $(M_{n_j}(\D_j), \sharp_j)$. 

For each $j$,  let $\mathfrak{a}_j$ 
be the image of  the algebra $M_{n_j}(\D_j)$  in the direct sum $\Bb$ under the canonical inclusion, so that 
 $\mathfrak{a}_j$ is a minimal two-sided ideal in $\Bb$.  If $\sharp$ does not preserve the direct sum structure of $\Bb$,  there is a $j$ such that
 $\sharp$ maps $\mathfrak{a}_j$ to $\mathfrak{a}_k$ for some $k\not=j$ (this follows from the minimality of the two-sided ideals $\mathfrak{a}_\ell$.) 
 Without loss of generality assume that $j=1, k=2$ (and therefore $n_1=n_2$, and $\D_1=\D_2$). For each positive integer $n$, consider the element $u_n\in \Bb$ given by
 \[ u_n=\left(nI_1, \frac{1}{n}I_2, I_3,\dots, I_N\right),\]
 where $ I_j$ is the identity of  $M_{n_j}(\D_j)$ . Then it follows that
 \[ u_n^\sharp=  \left(\frac{1}{n}I_1, nI_2, I_3,\dots, I_N\right) \]
 so that $u_nu_n^\sharp$ is the identity element of $\Bb$, and consequently $u_n\in \Gamma(\Bb,\sharp)$. It follows that $\Gamma(\Bb, \sharp)$ is not compact, but since $\phi$ is a $*$-isomorphism,
 we have that $\Gamma(\Aa,*)$ is also noncompact, which contradicts the assumption that  $(\Aa, *)$ is polarizable. 
 
Therefore we have $(\Bb,\sharp)=\bigoplus_{j=1}^N(M_{n_j}(\D_j), \sharp_j)$, a $*$-direct sum, where $\sharp_j$ denotes the restriction of $\sharp$ to $\mathfrak{a}_j$, which can be identified with $M_{n_j}(\D_j)$. Since $(\Bb, \sharp)$ is polarizable, it easily follows that each summand $(M_{n_j}(\D_j), \sharp_j)$ is also
polarizable. By Proposition~\ref{prop-noncompact}, it follows that each $\sharp_j$ is the conjugate transpose involution. 
		
\end{proof}

\subsection{Polarizability of matrix algebras} In this subsection, we will prove that 
a finite $*$-direct sum of matrix algebras over division algebras is polarizable. We note some properties of the vector-valued integral \eqref{eq-vectint} defined
by the condition \eqref{eq-intdefn}:
\begin{enumerate}[wide]
	\item If $f:G\to V$ is a continuous function on the group $G$  with values in a finite dimensional real vector space $V$,
	$W$ is another finite dimensional real vector space and $T:V\to W$ is an $\rl$-linear map, then we have:
	\begin{equation}
		\label{eq-linearity}T\left( \int_G fdg\right)= \int_G \left(T\circ f\right) \,dg 
	\end{equation}
	where now the integral on the right is the vector-valued integral of a $W$-valued function  on $G$. The relation \eqref{eq-linearity} is easily verified using the defining condition \eqref{eq-intdefn}.
	
	\item Again, let $f:G\to V$ be a continuous function on the group $G$  with values in a finite dimensional real vector space $V$. Then 
	we have 
	\begin{equation}\label{eq-span}
	\int_G fdg \in \Span_\rl \left(f(G)\right),
\end{equation}
where $f(G)=\{f(g): g\in G\}\subset V$ is the image of the map $f$. To see \eqref{eq-span}, let $W= \Span_\rl \left(f(G)\right) $,  and let $\pi:V\to V/W$ be the quotient map.  Then by \eqref{eq-linearity}, 
\[\pi\left(\int_G f dg\right)= \int_G \pi(f(g))dg=0,\]
which shows that $\int_G fdg\in W$. With a little more work, one can show that $\int_G fdg$ belongs to the convex hull of $f(G)$, but we do not need 
this. 
\end{enumerate}

 We also recall a few definitions. Let $\sigma\in S_n$ be a permutation, i.e., a bijection of $\{1,\dots, n\}$ with itself. We can associate with $\sigma$ an $n\times n$ \emph{permutation matrix} $P_\sigma$, where
the $j$-th column of $P_\sigma$ is $e_{\sigma(j)}$ for $1\leq j\leq n$, with $e_i$ denoting the $n\times 1$ column matrix with an $1$ in the $i$-th place and 0's everywhere else. It is clear that 
$P_\sigma e_j=e_{\sigma(j)}$. Given a group $H$ of permutation matrices, we say that it is \emph{transitive} if the corresponding permutations act transitively on the set $\{1,\dots, n\}$.

\begin{prop}	\label{prop-polargroupsMnA}
	Let $(\Aa,\sharp)$ be a $*$-algebra, and let $G\subset \Gamma(\Aa,\sharp)$ be a polarizing group. Let $n\geq 2$ and let 
	$\Delta_G\subset M_n(\Aa)$ be the collection of $n\times n$ diagonal matrices with the diagonal entries taken from $G$:
	\[ \Delta_G=\{\diag(g_1,\dots, g_n)\in M_n(\Aa): g_j\in G\},\]
	and let $H$ be a transitive group of permutation matrices in $M_n(\Aa)$. Then
	the set 
	\[ \Delta_GH= \{DP: D\in \Delta_G, P\in H\}\]
	is a compact subgroup of $(M_n(\Aa), *)$ which is polarizing, where $*$ is the involution induced on $M_n(\Aa)$ by the involution $\sharp$ of $\Aa$ 
	as in \eqref{eq-induced}.
\end{prop}

\begin{proof} It is clear that $\Delta_G\cong G^n$ is a  compact group as a Cartesian power  of the compact group $G$ and $\Delta_GH$ is compact as a union of finitely many homeomorphic 
	copies of $\Delta_G$. If $D\in \Delta_G$, we can write $D=\diag(g_1,\dots, g_n)$ for $g_j\in G, 1\leq j \leq n$, so that we have
	\[ D^*D= \diag(g_1^\sharp,\dots, g_n^\sharp)\cdot \diag(g_1,\dots, g_n)=\diag(1_\Aa,\dots, 1_\Aa),\]
	since $G\subset \Gamma(\Aa,\sharp)$ by hypothesis. Also, we have for a permutation matrix $P\in M_n(\Aa)$:
	\[ P^*P=(P^\sharp)^TP=P^TP=I.\]
	It follows that $\Delta_G H$ is a compact subset of the unitary group $\Gamma(M_n(\Aa),*)$.

	A direct computation shows that for each $\sigma\in S_n$, and $D\in  \Delta_G$ we have
	\begin{equation}
		\label{eq-semidirect}
		(P_{\sigma})^{-1} D P_\sigma =P_{\sigma^{-1}}D P_\sigma=D'\in \Delta_G
	\end{equation}
	where $D'$ is a diagonal matrix whose  nonzero entries are a permutation of the diagonal entries of $D$. More precisely, if $D=\diag(g_1,\dots, g_n)$ where $g_j\in G$,
	then $D'=\diag(g_{\sigma(1)},\dots, g_{\sigma(n)})$.

	The relation \eqref{eq-semidirect} shows not only that $\Delta_G H$ is a group, but also that $\Delta_G$ is normal subgroup of the group $\Delta_G H$. The intersection
	$\Delta_G\cap H$ consists of diagonal matrices which are permutation matrices and therefore is $\{I\}$.  It folows that $\Delta_G H$ is the (internal) semidirect product of the two subgroups 
	$\Delta_G$ and $H$. Since on a semidirect product, the Haar measure is equal to the product measure  (e.g. \cite[pp. 96ff.]{nachbin}), we have, using Fubini's theorem:
	\[ \int_{\Delta_GH}gdg= \int_{H} \int_{\Delta_G} D P dD dP =\left(\int_{\Delta_G} D dD \right)\cdot \left(\int_H PdP\right)= 0\cdot\left(\int_H PdP\right)=0, \]
	using the fact that $G$ is polarizable so \eqref{eq-gdg} holds.

	We now show that $\Delta_G H$ generates $M_n(\Aa)$. Let $E_{i,j}$ be the $n\times n$ matrix with $1_\Aa$ in the $i$-th row and $j$-th column and zeros everywhere else. First we show that $E_{i,j}$ is in the 
	subalgebra of $M_n(\Aa)$ generated by the group $\Delta_G H$. 
	For $1\leq j \leq n$ and $g\in G$, denote by $D_j(g)\in \Delta_G$ the diagonal matrix
	with a $1_\Aa$ in the $j$-th place and $g$ in each other place.  Then we clearly have, using \eqref{eq-gdg} that 
	$ E_{j,j}= \int_G D_j(g)dg.$
	%
	It follows by \eqref{eq-span}  that $E_{j,j}$ is in the algebra generated by $\Delta_G$.  Now observe that
	\[ E_{i,j}e_k = \begin{cases} e_i &\text{if  }     k = j \\ 0 & \text{otherwise} \end{cases}. \]
	Since $H$ is transitive, by definition there exists a $P \in H$ such that $Pe_j = e_i$, so that
	\[ P E_{j,j} e_k = \begin{cases} Pe_j= e_i &\text{if   } k = j \\ P(0)=0 & \text{otherwise} \end{cases}.\]
	Hence $PE_{j,j} = E_{i,j}$, so that $\Delta_G H$ generates all of the $E_{i,j}$. 
	
	
	Now we note that $\alpha I$ is generated by $\Delta_G H$ for each $\alpha \in \Aa$. Since $G$ generates $\Aa$,  there exist $g_1, \dots g_n \in G$ and $\lambda_1, \dots ,\lambda_n \in \rl$ so that $\alpha = \sum \lambda_i g_i$. Hence
	$\alpha I = \sum_{i=1}^n \lambda_i (g_iI) ,$
	where each $g_iI \in \Delta_G$.  Writing a matrix $A=(\alpha_{i,j})\in M_n(\Aa)$ as 
	\[ A = \sum_{i,j=1}^n (\alpha_{i,j}I)\cdot E_{i,,j} \]
	we see that $\Delta_G H $ is a polarizing subgroup of $M_n(\Aa)$.
\end{proof}
From this we deduce the following:
\begin{prop}\label{prop-kappa2}
	A $*$-algebra  which is $*$-isomorphic to a finite $*$-direct sum of matrix algebras over real division algebras, each endowed with the conjugate-transpose involution, is polarizable.
\end{prop}

\begin{proof}

From part (3) of Proposition~\ref{prop-structure1} it follows that a $*$-algebra is polarizable if and only if it has a compact subgroup of the unitary group whose $\rl$-linear span is the whole algebra. In particular, any $*$-algebra which admits 
a polarizing subgroup is polarizable.

Using this criterion it is easy to see that $\rl$, $\cx$ and $\hx$ are each polarizable if
endowed with the standard conjugation operation. 
Indeed, each of the finite groups $\{\pm 1\}, \{\pm 1, \pm i\} $ and  $\{\pm 1, \pm i, \pm j,\pm k\}$ consist of unitary elements in $\rl,\cx$ and $\hx$ respectively and span the respective division algebra.

Therefore, by Proposition~\ref{prop-polargroupsMnA}, the  $*$-algebra  $(M_n(\D), *)$, where the involution $*$ is the standard conjugate-transpose operation, is polarizable, where 
$\D$ is one of $\rl,\cx,\hx$. Notice also that the $*$-direct sum of a finite number of polarizable $*$-algebras is easily seen to be polarizable. Therefore the $*$-direct sum $\bigoplus_{j=1}^N (M_{n_j}(\D_j) *)$ is polarizable. Finally polarizability is clearly preserved by $*$-isomorphisms, so the result follows. 

\end{proof}
\section{Proof of Theorems~\ref{thm-main} and \ref{thm-kappa}}
\subsection{Integrals on   compact multiplicative groups} \label{sec-moments}
Let $\Aa$ be a $*$-algebra and let $G$ be a compact subgroup of $\Aa^\times$, the group of 
units (i.e. invertible elements) of  $\Aa$. For an integer
$k\geq 0$, we define the $k$-th \emph{moment} of $G$ to be the element of $\Aa$ given by
\begin{equation}\label{eq-momentdef}
	\mu_k(G)=\int_G g^k dg,
\end{equation}
where the integral is that of the  $\Aa$-valued function  $g\mapsto g^k$
  taken with respect to the normalized 
Haar measure of the compact group $G$, defined as in \eqref{eq-intdefn}. 	The relation \eqref{eq-linearity} will be used repeatedly without further comments in the computations below.  We now collect some basic information about moments.
\begin{prop}\label{prop-momentproperties}
	For each compact subgroup $G\subset \Aa^\times$ and each integer $k\ge0$, the element $\mu_k(G)$ belongs to the subalgebra $\Bb$ of $\Aa$ generated by $G$, and in fact
		 lies in the center of $\Bb$, i.e., for each $h\in \Bb$, we have $h\mu_k(G)=\mu_k(G)h$.
\end{prop}
\begin{proof}

	 Since $G$ being a group is closed under multiplication in $\Aa$, one sees easily that 
	$ \Bb= \Span_\rl G,$
	i.e. the algebra $\Bb$ generated by $G$ coincides with the linear span of the subset of $G$ of the $\rl$-vector space $\Aa$.  However the integrand of the vector-valued integral \eqref{eq-momentdef} takes values in the linear subspace $\Bb$ of $\Aa$, and therefore by \eqref{eq-span} the value of the integral lies in $\Bb$.  
	Now let $h\in G$, then
	 \[ h\mu_k(G)h^{-1}= \int_G hg^kh^{-1}dg = \int_G (hgh^{-1})^kdg=\int_G g^kdg=\mu_k(G), \]
	 using the invariance of the Haar measure. Therefore, we have for each $h\in G$ that 
	 $h\mu_k(G)=\mu_k(G)h$. Since $\Bb$ is the collection of real linear combinations of the elements of $G$, it follows that the same relation holds for $h\in \Bb$ as well.
\end{proof}
We now compute the first and second moments of a group  under appropriate hypotheses.

\begin{prop}\label{prop-firstmoment}
	Let $G\subset \Aa^\times$ be a compact subgroup and
	 suppose  that one of the following hold:
	\begin{enumerate}[label=(\alph*)]
		\item $-1_\Aa\in G$, \label{condition a}
		\item $G\not= \{1_\Aa\}$ and the subalgebra $\Bb$  of $\Aa$ generated by $G$ is simple  (i.e. has no nontrivial two sided ideals).\label{condition b}
	\end{enumerate}
Then we have
\[ \mu_1(G)=0.\]
\end{prop}
\begin{proof}
	If condition \ref{condition a} holds, then by the invariance of Haar measure
	\[ \mu_1(G)= \int_G(-1_\Aa)g dg = - \mu_1(G).\]
Now,  assume \ref{condition b}.  If $\Bb= \rl$, then by the nontriviality of $G$, we must have $G=\{\pm 1_\Aa\}$, so that $\mu_1(G ) = \frac{1}{2}(1-1)=0$.  Assuming therefore that
$\Bb\not =\rl$,  since by Proposition~\ref{prop-momentproperties}, the element $\mu_1(G)$ is in the center of 
$\Bb$ it follows that the set $\mathfrak{J}=\{h:h\mu_1(G)=0\}$ is a two sided ideal in $\Bb$. Since $\Bb\not=\rl$, there are two $\rl$-linearly 
independent elements $g_1,g_2\in G$. If we let $h=g_1-g_2$, we have $0 \not=h \in \mathfrak{J}$ since
\[ h\mu_1(G)= \int_G g_1gdg- \int_G g_2gdg= \int_G gdg-\int_G gdg=0.\]
By simplicity of $\Bb$, $\mathfrak{J}$ is all of $\Bb$.  It follows that $\mu_1(G)= 1_\Aa\cdot \mu_1(G)=1\mu_1(G)=0$.
\end{proof}
The following proposition summarizes the properties of the second moment that we will need in our  application:

\begin{prop} \label{prop-secondmoment}
	Let $(\Aa,*)$ be a $*$-algebra, and $G\subset \Gamma(\Aa,*)$ be a compact subgroup of unitary elements of $\Aa$ that generates $\Aa$. Then
for each $\alpha\in \Aa$, we have
		\begin{equation}
			\label{eq-alphastargamma}
	 \int_G g\alpha gdg=\alpha^* \mu_2(\Gamma(\Aa,*)).
	\end{equation}\end{prop}

\textbf{Remark:} Thanks to part (3) of Proposition~\ref{prop-structure1}, we know that  the unitary group $\Gamma(\Aa,*)$ is compact, and therefore $ \mu_2(\Gamma(\Aa,*))$ makes sense.
 Taking $\alpha= 1_\Aa$ in \eqref{eq-alphastargamma}, we have the remarkable fact that in a  polarizable $*$-algebra $(\Aa,*)$:
\begin{equation}\label{eq-mu2inv}
\mu_2(G)= \mu_2(\Gamma(\Aa,*)),
\end{equation}
for each compact subgroup $G$ of $\Gamma(\Aa,*)$ which generates the algebra $\Aa$. For a polarizable algebra $(\Aa, *)$ we will denote, by abuse of notation,
\[ \mu_2(\Aa)= \mu_2(\Gamma(\Aa,*)).\]
We refer to the element $\mu_2(\Aa)\in \Aa$ as the \emph{ second moment} of the polarizable algebra $(\Aa, *)$.
\begin{proof}
For each $h\in G$ we have, by an application of Haar invariance, that 
	\begin{equation*}
		 \int_{G} gh gdg = h^*h\int_{G} gh gdg = h^*\int_{G} (h g)^2 dg = h^*\int_{G} g^2 dg= h^*\mu_2(G),
	\end{equation*}
	using the fact that  $hh^*=1_\Aa$. Since $\Span_\rl{G} = \Aa$, we 
	can write each $\alpha \in \Aa$ as a finite sum  $\alpha = \sum \lambda_j h_j$ 
	for $h_1 ,\dots , h_n \in G$ and $\lambda_1 , \dots , \lambda _n \in \rl$. Hence
	\begin{equation}
		\label{eq-alphastargsquared}
		\int_{G} g \alpha g dg = \sum_{j=1}^n \lambda_j \int_{G} g h_j g dg = \sum_{j=1}^n \lambda_j h_j^* \int_{G} g^2 dg = \alpha^*\mu_2(G).
	\end{equation}
	To complete the proof we need to establish the relation~\eqref{eq-mu2inv}. This is obtained from the following computation:
	\begin{align*}
		\mu_2(G)&= \mu_2(G)\cdot\int_{\Gamma(\Aa,*)}d\gamma 
		=\int_{\Gamma(\Aa,*)}\gamma\gamma^* \mu_2(G)d\gamma \\
		&= \int_{\Gamma(\Aa,*)}\gamma\cdot\left(\int_G g\gamma g dg\right)d\gamma &\text{ using \eqref{eq-alphastargsquared}}\\
		& =\int_{\Gamma(\Aa,*)}\int_G(\gamma g)^2 dg d\gamma= \int_G\int_{\Gamma(\Aa,*)}(\gamma g)^2d\gamma dg\\
		&=\int_G\int_{\Gamma(\Aa,*)}\gamma^2d\gamma dg &\text{ by Haar invariance on $\Gamma(\Aa,*)$}\\
		&= \int_G \mu_2(\Gamma(\Aa,*))dg= \mu_2(\Gamma(\Aa, *)).
	\end{align*}

\end{proof}

 \label{sec-polarization}
\subsection{The key computation} We begin with the following consequence of the computations of Section~\ref{sec-moments}:
	\begin{prop}\label{prop-computation}
	Let $\Aa$ be a $*$-algebra and let $G\subset \Gamma(\Aa,*)$ be a compact subgroup which generates the algebra $\Aa$. Then for each 
	left $\Aa$-module $\X$, and each Hermitian form $Q:\X\times \X\to \Aa$, we have
	\begin{equation}\label{eq-prepol}
 \int_G q(x+gy)gdg= (q(x)+q(y))\cdot \mu_1(G)+ Q(x,y)\cdot (1_\Aa+\mu_2(\Gamma(\Aa,*))),
	\end{equation}
where $x,y\in \X$, and we set 
	\[q(x)= Q(x,x), \quad x\in \X.\]
\end{prop}
\begin{proof}For $x,y\in \X$ and $g\in G\subset \Aa$ we have 
	\begin{align*}
		q(x+gy)&= Q(x+gy, x+gy)\\
		&= Q(x,x)+Q(x,gy)+ Q(gy,x)+Q(gy,gy)\\
		&=q(x)+ Q(x,y)g^*+ g Q(y,x)+ g q(y)g^*.
	\end{align*}
	Therefore, multiplying on the right by $g$ and integrating on $G$ with respect to Haar measure we have
	\[  \int_G q(x+gy)gdg =A+B+C+D,\]
	where
	\[A = \int_G q(x)gdg = q(x)\mu_1(G),\]
	\[ B= \int_G Q(x,y)g^* g dg = \int_G Q(x,y)dg= Q(x,y)\int_G dg=Q(x,y),\]
	\begin{align*}
		C&= \int_G gQ(y,x)gdg=Q(y,x)^*\int_G g^2dg= Q(y,x)^*\mu_2(\Gamma(\Aa,*)) & \text{by  \eqref{eq-alphastargamma} and \eqref{eq-mu2inv}}\\
		&= Q(x,y) \mu_2(\Gamma(\Aa,*)).
	\end{align*}
	and 
	\[ D= \int_G g q(y)g^*gdg = \left(\int_G gdg\right) q(y)= \mu_1(G)q(y)=q(y)\mu_1(G), \]
	where in the last step we use the fact that $\mu_1(G)$ lies in the center of  $\Aa$ (see  Proposition~\ref{prop-momentproperties} above). 
	Combining these equations, the result follows. 
\end{proof}

\subsection{Second moments of matrix algebras}
\begin{prop}\label{prop-2ndmoment}
	Let $(\Aa,\sharp)$ be a polarizable algebra,  and endow the matrix algebra $M_n(\Aa)$ with the involution $*$ given by  $A\mapsto A^*= (A^\sharp)^T$ induced by $\sharp$ as 
	in \eqref{eq-induced}. Then by Proposition~\ref{prop-polargroupsMnA}, 
	the $*$-algebra $(M_n(\Aa), *)$ is polarizable. The second moment of this algebra is
	given by:
	\begin{equation}
		\label{eq-gammaMnA}
		\mu_2(M_n(\Aa)) = \dfrac{1}{n}\cdot\mu_2(\Aa)I,
	\end{equation} 
where $I$ is the identity matrix of $M_n(\Aa)$.
\end{prop}

\begin{proof} 
	Let $\Delta := \Delta_{\Gamma(\Aa,\sharp)}$ be the collection of diagonal matrices in $M_n(\Aa)$ with entries from the unitary group $\Gamma(\Aa,\sharp)$ 
	as in Proposition~\ref{prop-polargroupsMnA}. Let $H$ be the cyclic group of permutation matrices in $M_n(\Aa)$ corresponding to the $n$-cycle 
	$\sigma=(1,2,\dots, n)\in S_n$,
	i.e.
	\[H=\{P_{\sigma^k}: 0\leq k \leq n-1 \},\]
	where for $\tau\in S_n$, $P_\tau$ is the permutation matrix such that $P_\tau e_i=e_{\tau(i)}$, where $e_j$ is the column vector with 1 in the $j$-th row and zeroes 
	everywhere else.

	 A computation analogous to that in \eqref{eq-semidirect} shows that for $D=\diag(g_1,\dots,g_n)\in \Delta$  (with $g_j\in \Gamma(\Aa,\sharp)$)  and $P_{\sigma^k}\in H$ we have
	 \[ DP_{\sigma^k}D P_{\sigma^k}=P_{\sigma^{2k}}D'' D'\]
	 where
	 \[ D'=\diag\left(g_{\sigma^k(1)},\dots, g_{\sigma^k(n)}\right), \quad D''=\diag\left(g_{\sigma^{2k}(1)}, \dots, g_{\sigma^{2k}(n)}\right)\in \Delta_G.\]
	Therefore we have (with $e_1^T$ the row vector with 1 in the first slot and zeroes everywhere else):
	\begin{align*}
		e_1^TDP_{\sigma^k}D P_{\sigma^k}e_1&= e_1^T P_{\sigma^{2k}}D'' D'e_1\\
&=		e_1^T P_{\sigma^{2k}}D'' g_{\sigma^k(1)} e_1\\
	&=	e_1^T P_{\sigma^{2k}} g_{\sigma^{2k}(1)}g_{\sigma^k(1)} e_1\\
	&= e_1^T g_{\sigma^{2k}(1)}g_{\sigma^k(1)}e_{\sigma^{2k}(1)}\\
	&= \begin{cases}
		 g_{1}g_{\sigma^k(1)}& \text{ if $\sigma^{2k}=\mathrm{id}$}\\
		 0 &\text{otherwise}.
	\end{cases}
	\end{align*}

	 By Proposition~\ref{prop-momentproperties} the element  $\mu_2(M_n(\Aa))$ is in the center of $M_n(\Aa)$. It is not difficult to see that 
the center of $M_n(\Aa)$ consists of matrices of the form $zI$, where $z$ belongs to 
the center of $\Aa$ and $I$ is the identity matrix of $M_n(\Aa)$.
Therefore there is a $z$ in the center of $\Aa$  such that $\mu_2(M_n(\Aa)) = z I$. We have 
\begin{align*} z&= e_1^T \mu_2(M_n(\Aa))e_1=\int_H \int_\Delta e_1^T DPDPe_1 dDdP\\
	&=  \frac{1}{n} \sum_{k=0}^{n-1}\int_\Delta e_1^T D P_{\sigma^k} D P_{\sigma^k}e_1 dD = \frac{1}{n} \sum_{\substack{\sigma^{2k}=\mathrm{id}\\0 \leq k \leq n-1}}\int_\Delta  g_{1}g_{\sigma^k(1)}dD\\
&=	\frac{1}{n }\sum_{\substack{\sigma^{2k}=\mathrm{id}\\0 \leq k \leq n-1}}\int_{\Gamma(\Aa, \sharp)} \dots \int_{\Gamma(\Aa, \sharp)} g_{1} g_{\sigma^k(1)}dg_1 \dots dg_n\nonumber \\
	&= \dfrac{1}{n} \allof{ \int_{\Gamma(\Aa, \sharp)} \dots \int_{\Gamma(\Aa, \sharp)} g_1^2 dg_1 \dots dg_n + \sum_{\substack{\sigma^{2k}=\mathrm{id}\\1 \leq k \leq n-1}}\int_{\Gamma(\Aa, \sharp)} \dots \int_{\Gamma(\Aa, \sharp)} g_{1} g_{\sigma^k(1)}dg_1 \dots dg_n } \nonumber\\
	&= \dfrac{1}{n} \allof{ \int_{\Gamma(\Aa, \sharp)} g_1^2 dg_1 +  \sum_{\substack{\sigma^{2k}=\mathrm{id}\\1 \leq k \leq n-1}}\int_{\Gamma(\Aa, \sharp)} \int_{\Gamma(\Aa, \sharp)} g_{1} g_{\sigma^k(1)}dg_1 dg_{\sigma^k(1)}} \nonumber\\
	& = \dfrac{1}{n}\allof{ \mu_2(\Aa) + \sum_{\substack{\sigma^{2k}=\mathrm{id}\\1 \leq k \leq n-1}}\allof{\int_{\Gamma(\Aa, \sharp)} g_1 dg_1} \allof{\int_{\Gamma(\Aa, \sharp)} g_{\sigma^k(1)} dg_{\sigma^k(1)}}} \nonumber \\ 
	&=\frac{1}{n}\mu_2(\Aa),
\end{align*}
using in the last line the property \eqref{eq-gdg} of polarizable subgroups such as $\Gamma(\Aa, \sharp)$.
\end{proof}
To use the above, we need to know the second moment of the algebra $(\Aa, \sharp)$. We now do this for the division algebras.
\begin{prop}
	Let $\D$ be a real division algebra, thought of as a $*$-algebra with the standard conjugate-transpose involution.  Then we have
	\begin{equation}\label{eq-mu2d}
		 \mu_2(\D)=\frac{2}{\delta}-1,\quad \delta=\dim_\rl\D=1,2, \text{or } 4.
	\end{equation}
Consequently, for the $*$-algebra $(M_n(\D),*)$, we have
\begin{equation}\label{eq-mu2mnd}
\mu_2(M_n(\D))= \frac{1}{n}\left(\frac{2}{\delta}-1	\right)I_n,
\end{equation}
where $I_n$ is the $n\times n$ identity matrix. 
\end{prop}
\begin{proof}
We use the relation \eqref{eq-mu2inv} to compute the second moments in \eqref{eq-mu2d}. When $\D=\rl$, the trivial multiplicative  subgroup $G=\{1\}$ is compact, contained in the unitary group, and generates the algebra $\D$ over $\rl$, so we have
\[\mu_2(\rl)= \mu_2(G)= 1= \frac{2}{\dim_\rl \rl}-1.\]

When $\D=\cx$, we can take $G$ to be the group $\{\pm 1, \pm i\}$ fourth roots of 1. Then we have
\[ \mu_2(\cx)= \frac{1}{4}\left( 1^2+(-1)^2+ i^2+ (-i)^2\right)= \frac{1}{4}(1+1+(-1)+(-1))=0=\frac{2}{\dim_\rl \cx}-1\]
verifying the assertion. 

When $\D=\hx$, the 8 element ``quaternion group" $G=\{\pm 1, \pm i, \pm j, \pm k\}$ 
is compact and spans $\hx$ over $\rl$. We therefore have
\begin{align*}
	\mu_2(\hx)&=\frac{1}{8} \left(1^2+(-1)^2+i^2+(-i)^2+j^2+(-j)^2+k^2+(-k)^2\right)\\
	&= \frac{1}{8}(1+1-1-1-1-1-1-1)= -\frac{1}{2}= \frac{2}{\dim_\rl \hx}-1,
\end{align*}
verifying \eqref{eq-mu2d} in all cases. The formula \eqref{eq-mu2mnd} now follows using Proposition~\ref{prop-2ndmoment}.
\end{proof}

\subsection{End of proof of Theorems~\ref{thm-main} and \ref{thm-kappa}}
We have now assembled all the ingredients to prove these two results, which we will prove
simultaneously.  As in Theorem~\ref{thm-main}, let  $(\Aa,*)$ be a polarizable $*$-algebra,
and let $G\subset \Gamma(\Aa,*)$ be a polarizing subgroup, and let $Q, q$  and $\X$ also
 have the same meaning as in the statement of that theorem. By hypothesis \eqref{eq-gdg}
we have, using the computation \eqref{eq-prepol}, that for $x,y\in \X$ we have:
\begin{equation}\label{eq-prepol1}
\int_G q(x+gy)gdg=  Q(x,y)\cdot (1_\Aa+\mu_2(\Aa)), 
\end{equation}
so to complete the proof of Theorem~\ref{thm-main} we need to show that the element 
\begin{equation}\label{eq-1plusmu}
	1_\Aa + \mu_2(\Aa)	\end{equation}
of $\Aa$ is a unit of the ring $\Aa$, i.e., it is invertible, and its inverse  is the polarization constant $\kappa$ of the algebra $(\Aa,*)$, i.e.
\begin{equation}
	\label{eq-kappa3}
	\kappa=\kappa(\Aa,*)= (1_\Aa + \mu_2(\Aa))^{-1}.
\end{equation}

By Proposition~\ref{prop-kappa1}, the $*$-isomorphism  $\phi$ of \eqref{eq-isom} exists if $(\Aa,*)$ is polarizable, and by Proposition~\ref{prop-kappa2}
if $\phi$ exists, then $(\Aa,*)$ is polarizable.
 To complete the proof of Theorem~\ref{thm-kappa}, we  need to show that the multiplicative inverse of the element \eqref{eq-1plusmu} is given by the formula  \eqref{eq-kappacomp}.

Thanks to the invariance of the problem under $*$-isomorphisms, it is clear that we only need to consider the case where
$(\Aa,*)$ is already a $*$-direct sum of the matrix algebras $(M_n(\D_j), *)$ where the involution is the conjugate-transpose involution, i.e. the map $\phi$ of \eqref{eq-isom}  
is the identity. 
In this case the element \eqref{eq-1plusmu} is given by 
\begin{align*}
	1_\Aa + \mu_2(\Aa)&=		1_\Aa +\left( \mu_2(M_{n_1}(\D_1))I_1, \dots, \mu_2(M_{n_N}(\D_{N})I_N)\right)\\
	&= \left( (1+ \mu_2(M_{n_1}(\D_1))I_1, \dots ,(1+\mu_2(M_{n_N}(\D_N))I_N\right)\\
	&= \left( \left(1+ \frac{1}{n_1}\left(\frac{2}{\delta_1}-1	\right)\right)I_1, \dots ,\left(1+ \frac{1}{n_N}\left(\frac{2}{\delta_N}-1	\right)\right)I_N \right).
\end{align*} 
Now since for each $j$ we have $\delta_j\leq 4$ and $n_j\geq 1$, therefore
\[ 1+ \frac{1}{n_j}\left(\frac{2}{\delta_j}-1	\right)\geq \frac{1}{2}.\]
It follows that the element \eqref{eq-1plusmu} is invertible, thus completing the proof of Theorem~\ref{thm-main}. It also follows that  its inverse is
\[ (1_\Aa+\mu_2(\Aa))^{-1}=\left( \frac{n_1\delta_1}{(n_1-1)\delta_1 +2 }I_{n_1}, \dots,\frac{n_j\delta_j}{(n_j-1)\delta_j +2 }I_{n_j},\dots, \frac{n_N\delta_N}{(n_N-1)\delta_N+2 }I_{n_N}\right), \]
completing the proof of Theorem~\ref{thm-kappa}.

\section{Quadrances and Hermitian forms} \label{sec-jvn}

\subsection{Quadrances arising from Hermitian forms}
\begin{prop}\label{prop-diag-rest}
	If $Q$ is a Hermitian form on a left $\Aa$-module $\X$, then the diagonal restriction
	\[ q:\X\to \Aa, \quad q(x)=Q(x,x), \text{ where } x\in \X\]
	is a quadrance, and the \emph{generalized parallelogram identity holds}:
 for a compact subgroup  $H\subseteq \Aa^\times$  such that
 \begin{enumerate}[label=(\alph*)]
 	\item $q(hx)=q(x)$ for all $h\in H$, and 
 	\item $\mu_1(H)=0$
 \end{enumerate}
hold, 	we have  for $x,y \in X$,
	\begin{equation}
		\label{eq-genpar}q(x)+q(y)= \int_H q(x+hy)dh.
	\end{equation}
	
\end{prop}

Notice that any real $*$-algebra always has a subgroup $H$ of $\Aa^\times$ that satisfies the hypotheses of the above proposition. This is 
the two element subgroup $H=\{\pm 1_\Aa\}$ for which \eqref{eq-genpar} 
becomes the \emph{classical parallelogram identity} \eqref{eq-classpar}
which reduces to \eqref{eq-classic-parallelogram} for 
inner product spaces over $\cx$. 

\begin{proof}

	 For  $x \in \X,\alpha \in \Aa$ we have $Q(x,x)^* = Q(x,x)$, i.e. 
		$q(x)^* = q(x),$
		and $Q(\alpha x, \alpha x) = \alpha Q(x,x)\alpha^*$, i.e.
		$ q(\alpha x) = \alpha q(x) \alpha^*,$ thus verifying the first two conditions in the definition of a quadrance (Definition~\ref{defn-quadrance} above).
		For  $x,y, \in \X$, $\lambda\in \rl=\rl\cdot 1_\Aa, \alpha\in \Aa$, we have
		\begin{equation*}
			\begin{split}
				q(\lambda  x + \alpha y) &= Q(\lambda  x + \alpha y,\lambda  x + \alpha y) \\
				&=Q(\lambda  x,\lambda  x) = Q(\lambda  x, \alpha y) + Q(\alpha y, \lambda  x) + Q(\alpha y, \alpha y) \\
				&=\lambda^2  q(x)  + \lambda  Q(x,y) \alpha^* + \alpha Q(y,x) \lambda  + \alpha q(y)\alpha^*.
			\end{split}
		\end{equation*}
		Since multiplication in $\Aa$ is continuous, it follows that the mapping \eqref{eq-continuity} is continuous.  This shows that $q$ is a quadrance.
		
		For $x,y \in X$ a  computation yields \eqref{eq-genpar}:
		\begin{equation*}
			\begin{split}
				\int_H q(x+hy)dh &= \int_H Q(x+hy,x+hy)dh \\
				&= \int_H q(x)dh + \int_H q(hy)dh + \int_H Q(x,hy)dh + \int_H Q(hy,x)dh \\
				&= \int_H q(x)dh + \int_H q(y)dh + Q(x,y)\int_H h^*dh + \allof{\int_H hdh}Q(y,x) \\
				&=q(x) + q(y),
			\end{split}
		\end{equation*}
		where 
		$\displaystyle{\int_H h^*dh= \allof{\int_H h dh}^* = 0}$
			by $\rl$-linearity of $*$.

\end{proof}

\textbf{Remark:} It would be very interesting to know if the classical parallelogram identity 
\eqref{eq-classpar} can be replaced in Theorem~\ref{thm-jvn} by a generalized parallelogram identity with respect to some other compact  group $H\subset \Aa^\times$. 

	\subsection{Polarization of quadrances and proof of Theorem~\ref{thm-jvn}}
The proof of Theorem~\ref{thm-jvn} consists of applying the operation on the right hand side of \eqref{eq-polarization} to the given quadrance, and showing that the resulting function is a Hermitian form. The following is the first step in the argument:

\begin{prop}
	\label{prop-qgproperties}
	Let $(\Aa,*)$ be a  polarizable $*$-algebra and let $\X$ be a left $\Aa$-module. Let $q \from \X \to \Aa$ be a quadrance on $\X$ and  let $G\subset \Gamma(\Aa)$ be a polarizing subgroup.  Define 
	$Q_G \from \X \times \X \to \Aa$ by
	\begin{equation}\label{eq-Qg}
		Q_G(x,y) = \kappa\cdot\int_G q(x + gy)gdg,
	\end{equation}
where $\kappa$ is the polarization constant of $(\Aa,*)$.
	Then $Q_G$ satisfies the following properties, for all $x,y\in \X$ :
	\begin{enumerate}[label=(\alph*)]
		\item $Q_G(x,x)=q(x).$
		\item $Q_G(x,y)^* = Q_G(y,x)$.
\label{item-hermitian}
		\item $Q_G(hx,y) = h Q_G(x,y)$  for all $h\in G$.
				\item $Q_G(x,0) = 0$.
	\end{enumerate}
\end{prop}
\begin{proof}
	\begin{enumerate}[label=(\alph*), wide]
		\item We compute 
		\begin{align*}
				Q_G(x,x) &=   \kappa\cdot\int_G q((1+g)x)gdg = \kappa\cdot\int_G (1+g)q(x)(1+g)^*gdg \\
				&= \kappa\cdot \allof{q(x)\int_G gdg + q(x)\int_G dg + \int_G gq(x)gdg + \allof{\int_G g dg} q(x)} \\
				&= \kappa\cdot \allof{q(x) + q(x)^*\mu_2(\Aa)} =\kappa\cdot(1_\Aa+\mu_2(\Aa)) q(x) = q(x),
		\end{align*}
		using the facts that $q(x)^*=q(x)$ and $\mu_1(G)=0$,
		
	\item We see from the formula \eqref{eq-kappacomp} that 
		the polarization constant $\kappa$ satisfies $\kappa^*=\kappa$  and lies in the center of  $\Aa$. Therefore we 
		can compute:
		\begin{equation*}
			\begin{split}
				Q_G(x,y)^* &=  \kappa\cdot\int_G (q(x+gy)g)^* dg = \kappa\cdot\int_G g^*q(x+gy) dg \\
				&= \kappa\cdot\int_G q(g^*x+y) g^* dg = \kappa\cdot \int_G q(gx+y)g dg \\
				&=Q_G(y,x).
			\end{split}
		\end{equation*}

		\item This follows by an application of Haar invariance, and the properties of $\kappa$ used in the previous part:
		\begin{equation*}
			\begin{split}
				Q_G(hx,y) &=\kappa\cdot \int_G q(hx+gy)gdg = \kappa\cdot\int_G h q(x+h^*gy)h^*g dg \\
				&= \kappa\cdot h \int_G q(x + h^*gy)h^*gdg = h \kappa\cdot\int_G q(x+gy)gdg \\
				&= hQ_G(x,y).
			\end{split}
		\end{equation*}
			\item We have:  $\displaystyle{Q_G(x,0) = \kappa\cdot \int_G q(x)gdg = 0.}$
	\end{enumerate}
\end{proof}

\begin{proof}[Proof of Theorem~\ref{thm-jvn}] Let $G$ be a polarizing subgroup of $\Gamma(\Aa,*)$ (for example, we could take 
	$G$ to be $\Gamma(\Aa,*)$ itself). Define the map $Q_G:\X\times \X\to \Aa$ by the formula
	\eqref{eq-Qg}.	
	We will prove that
	$Q_G$ is a Hermitian form on $\X$;  by  part (a)  of Proposition \ref{prop-qgproperties}, the result would follow if we can also show uniqueness:  if $Q$ is a  Hermitian form such that 
	$Q(x,x) =q(x)$, then we have $Q=Q_G$. But by the polarization identity \eqref{eq-polarization} of Theorem~\ref{thm-main} we have for $x,y\in \X$ that
	\[ Q(x,y) = \kappa \cdot\int_G q(x+gy)gdg = Q_G(x,y). \]

	Thanks to the properties of $Q$ already established in Proposition~\ref{prop-qgproperties},
	 it is enough to show that $Q_G$ is $\Aa$-linear in the first argument. 
For any $x,y \in \X$ we have
	\begin{align*}
			Q_G(2x,y) &=\kappa\cdot\int_G q(x+x+gy)gdg \\
			&= \kappa\cdot\int_G \allof{2\allof{q(x) + q(x+gy)}-q(gy)}gdg \\
			&= \kappa\cdot \allof{2 \int_G q(x)gdg + 2\int_G q(x+gy)gdg - \int_G g q(y) dg}\\
			&= 2\kappa\cdot\int_G q(x+gy)gdg \\
			&= 2Q_G(x,y),
	\end{align*}
where we have used the parallelogram identity~\eqref{eq-classpar} to get the second equality. Using part \ref{item-hermitian} of Proposition~\ref{prop-qgproperties} we have:
	\begin{equation}
		\label{2hom}
		Q_G(x,2y) = Q_G(2y,x)^* = 2Q_G(y,x)^* = 2Q_G(x,y).
	\end{equation}
	Now, for $x,y,z \in \X$ we have using the classical  parallelogram identity ~\eqref{eq-classpar} that 
	\begin{equation*}
		\begin{split}
			Q_G(x,z) + Q_G(y,z)
			&=\kappa\cdot\int_G \allof{q(x+gz) + q(y + gz)}gdg \\
			&= \kappa\cdot\int_G \dfrac{1}{2}\allof{q(x+y + 2gz) + q(x-y)}gdg \\
			&= \dfrac{1}{2} \kappa\cdot\int_G q(x+y+ g(2z))gdg + \dfrac{1}{2}\kappa\cdot q(x-y)\cdot\int_G  gdg\\
			&= \dfrac{1}{2}Q_G(x+y,2z),
		\end{split}
	\end{equation*}
	and so, by \eqref{2hom}, it follows that $Q_G(x,z) + Q_G(y,z) = Q_G(x+y,z)$ for all $x,y,z \in X$.
From this we deduce easily that $Q_G(\cdot, z)$ is $\Q$-homogeneous, i.e. 
for $\lambda\in \Q$ we have
\begin{equation}\label{eq-homogeneity}
Q(\lambda x,y)=\lambda Q(x,y).
\end{equation}
To extend \eqref{eq-homogeneity} to the situation when $\lambda\in\rl$, fix
 $x,y \in \X$ and $\lambda\in \rl$ and let $\{\lambda_n\}$ be a sequence in $\Q$ converging to $\lambda$. Let$\{F_n\}$ be the sequence of $\Aa$-valued functions on $G$ given by 
$F_n(g)=q(\lambda_nx+gy)g$. Since the $\Aa$-valued function $(\xi, g)\mapsto q(\xi x+gy)g$ is uniformly continuous on the compact space $[\lambda-1, \lambda+1]\times G$, it follows that 
$F_n\to F$ uniformly on $G$ where $F(g)=q(\lambda x+gy)g$. Therefore
\begin{align*}
	\lambda Q_G(x,y) & = \lim_{n \to \infty} \lambda_n Q_G(x,y) = \lim_{n \to \infty}Q_G(\lambda_n x,y) 
		 = \kappa\cdot\lim_{n \to \infty} \int_G q(\lambda_n x + gy)gdg\\& = \kappa\cdot\int_G q(\lambda x + gy)gdg\quad \text{by uniform convergence}\\& = Q_G(\lambda x,y).
\end{align*}
	Finally, letting $\alpha = \sum_{i=1}^m \lambda_i g_i \in \Aa $ for some $\lambda_1 ,\dots, \lambda_m \in \rl$ and $g_1,\dots , g_m \in G$ we have
	\[ Q_G(\alpha x,y) = Q_G( \sum_{i=1}^m \lambda_j g_j x ,y) = \sum_{i=1}^m \lambda_ig_i Q_G(x,y) = \alpha Q_G(x,y) \]
	completing the proof that $Q_G$ is a Hermitian form. We have already proved uniqueness above. 
\end{proof}

\section{Some Examples}
\subsection{The three  real division algebras} \label{sec-divisionalgebras}We illustrate Theorems~\ref{thm-main} and \ref{thm-kappa} by classifying all the polarizing subgroups of $\rl, \cx$ and $\hx$, and therefore obtaining all polarization formulas for Hermitian forms defined on vector spaces over these scalars. By \eqref{eq-mu2d},
the polarization constants of the algebras (made into $*$-algebras with the usual conjugation operations), are
\begin{equation}\label{eq-kappa2}
 \kappa(\D)= \left(1+\mu_2(\D)\right)^{-1}= \left(1+\left(\frac{2}{\dim_\rl(\D)}-1	\right) \right)^{-1}= \frac{1}{2}\dim_\rl \D=\begin{cases}\frac{1}{2} & \text{ if $\D=\rl$}\\1 & \text{ if $\D=\cx$}\\2 &\text{ if $\D=\hx$}.
\end{cases} \end{equation}

	\subsubsection{The real numbers} In this case the unitary group is $\Gamma(\rl, \mathrm{id})=\{ x\in \rl: x^2=1\}=\{\pm 1\}$. Therefore the only polarizing subgroup of the reals is the group $\{\pm1\}$. The
	polarization formula for a  Hermitian form $Q$ on a real vector space is (which in this situation is simply a bilinear form):
\begin{align*}
	 Q(x,y)&= \kappa\int_{\{\pm1\}} q(x+gy)g dg= \frac{1}{2}\left(\frac{1}{2}\cdot q(x+y)+\frac{1}{2}\cdot q(x-y)\cdot(-1)\right)\\&= \frac{1}{4}\left(q(x+y)-q(x-y)\right). 
\end{align*}
	Of course, at the bottom, after diagonalizing the bilinear form, this is nothing but the identity $xy= \frac{1}{4}\left((x+y)^2-(x-y)^2\right)$ for real numbers.
	\subsubsection{Locus classicus: complex numbers} For $(\cx, \mathrm{conj})$, the unitary group is $\Gamma(\cx, \mathrm{conj})=\{z\in \cx: \abs{z}=1\}=U(1)$. Any compact subgroup of 
	this one dimensional Lie group is either a finite subgroup, or the whole of $U(1)$. It is easy to see that a finite subgroup of $U(1)$ is a cyclic group $G_N=\{\omega^j: 0\leq j\leq N-1\}$ 
	generated by the primitive $N$-th root  of unity $\omega=\exp\left(\frac{2\pi i}{N}\right)$. Notice that  $\Span_\rl G_N=\cx$ if and only if $N\geq 3$. It is clear also that  $\Span_\rl U(1)=\cx$, 
	and thanks to part (b) of Proposition~\ref{prop-firstmoment}, condition \eqref{eq-gdg} holds for both $G_N$ and $U(1)$.  It follows that any polarizing subgroup of $\cx$ is either $U(1)$ or 
	$G_N$ for $N\geq 3$. 
	
	For the group $G_N$, we obtain the ``$N$-th root polarization identity" (\cite[p. 12, Problem 1.10]{young} or \cite[pp. 53--54]{dangelo})
	\begin{equation}\label{eq-pol-c}
		Q(x,y)= \frac{1}{N} \sum_{k=0}^{N-1} q(x+\omega^ky)\omega^k, \quad \omega= e^{\frac{2\pi i}{N}}, \text{ where }  N\geq 3, 
	\end{equation}
	of which  \eqref{eq-pol-classical}  is the special case for $N=4$  and $Q=\ipr{\cdot, \cdot}$. For the polarizing group $G=U(1)$ we obtain the ``integral polarization identity"
	\begin{equation}\label{eq-pol-int}
		Q(x,y)= \frac{1}{2\pi}\int_0^{2\pi} q(x+e^{i\theta}y)e^{i\theta}d\theta,
	\end{equation}
which  can also be obtained from \eqref{eq-pol-c} as the limit when $N\to \infty$, by interpreting its right hand side as a Riemann sum of the integral on the right hand side of \eqref{eq-pol-int}. Notice that this passage to the limit is possible only because in the generalized polarization identity \eqref{eq-polarization}, the constant $\kappa$ depends only on the
algebra of scalars $\Aa$ and not on the polarizing group $G$.

\subsubsection{The quaternions} Just as with the complex numbers, there is an abundant 
supply of polarizing subgroups in the quaternions:
\begin{prop}
	A  subgroup of $\hx^\times$  is  polarizing if and only if it is compact and nonabelian. 
\end{prop}
\begin{proof} 
	Let $G$ be a compact subgroup of $\hx^\times$. The map $\hx^\times\to \rl^+$ given by
	$x \mapsto \norm{x}$ is a continuous homomorphism of the group of nonzero quaternions under multiplication to the group of positive real numbers under multiplication,
	and therefore the image of a compact subgroup $G$ of $\hx\setminus\{0\}$ is a compact subgroup of $\rl^+$, i.e., $\{1\}$, so that $G\subset Sp(1)$.

Now  suppose  that $G$ is polarizing in $\hx$, so in particular  $G$ generates $\hx$.
	Since $\hx$ is not commutative, it follows that $G$ must be nonabelian.
	
	For the converse assume that $G$ is nonabelian, and let $\Bb$ be the subalgebra of $\hx$ generated by the group $G$. Then $\Bb$ (with the involution induced from $\hx$)  is noncommutative and polarizable, since it is generated by a  compact nonabelian subgroup $G\subset \Gamma(\Bb,\mathrm{conj}) $ of
	unitary elements (see Proposition~\ref{prop-structure1} ).  Then by Theorem~\ref{thm-kappa}  the algebra $\Bb$ is isomorphic to $\bigoplus_{j=1}^N M_{n_j}(\D_j)$ for some real division algebras $\D_j$ and positive integers $n_j$. Therefore,
	if $\delta_j=\dim_\rl(\D_j)\in \{1,2,4\}$ we will have
	${ 1\leq \sum_{j=1}^N \delta_j n_j^2 \leq 3.}$
	It follows that $\Bb$ is isomorphic as an algebra to one of
	\[ \rl, \rl\oplus \rl, \cx, \rl\oplus \rl\oplus\rl, \rl\oplus \cx,\]
	each of which is commutative. Since $\hx$ is not commutative, this shows that $\Bb=\hx$.
		Since $\hx$, being a division ring, is simple, the condition \eqref{eq-gdg} now follows from part (b) of Proposition~\ref{prop-firstmoment}.

\end{proof}

Therefore the task of classifying all polarizing subgroups of 
the quaternions (and therefore finding all quaternionic polarization identities) is reduced to the classification of 
closed nonabelian subgroups of $Sp(1)$. Notice that
if $G$ is a polarizing subgroup in a noncommutative polarizable algebra $(\Aa, *)$, then any subgroup $H$ of the unitary group which is conjugate to $G$ (i.e., there is a
 $\gamma\in \Gamma(\Aa,*)$ such that $H=\gamma G \gamma^{-1}$) is clearly also polarizing. Therefore it will be sufficient to classify polarizing subgroups up to conjugacy.

\begin{prop}\label{prop-quaternions}
	A subgroup of $\Gamma(\hx, \mathrm{conj})=Sp(1)$  is  polarizing if and only if it is  one of the following:
	\begin{enumerate}
				\item The full group $Sp(1)$.
				\item A subgroup conjugate to the group
				\begin{equation}
					\label{eq-onedim}
					\{ \cos \theta + i\sin \theta : 0\leq \theta \leq 2\pi\}\cup \{j\cos \phi  + k\sin \phi  : 0 \leq\phi \leq 2 \pi \},
				\end{equation}
			 a Lie group isomorphic to $O(2)$. 
		\item A finite subgroup of $\hx$ 
		conjugate to one of the groups in the following table, 
		where 
		\[ \omega = \frac{1}{2}\left(-1+i+j+k	\right),\]
		and the last row corresponds to an infinity of groups, one for each integer $n\geq 2$. 
		\begin{center}

			\renewcommand{\arraystretch}{1.5} 
			\emph{
				\begin{tabular}{|l|c|c|}
					\hline
					Traditional name&Order  & Generators \\
					\hline\hline
					Binary tetrahedral (2$A_4$)	&24 &$\langle \omega, i\rangle$  \\
					\hline
					Binary octahedral (2$S_4$)	& 48 &$\left\langle \omega, \frac{1}{\sqrt{2}}(j+k)\right\rangle$  \\
					\hline
					Binary icosahedral (2$A_5$)	& 120 &$\left\langle \omega, \frac{1}{2}i + \frac{\sqrt{5}-1}{4}j+\frac{\sqrt{5}+1}{4} k\right\rangle$   \\
					\hline
					Binary dihedral	(2$D_{2n}$)	&$4n$  &$\left\langle \cos \left(\frac{\pi}{n} \right)+ \sin \left(\frac{\pi}{n}\right)i, j \right\rangle$  \\
					&($n\geq 2 $ an integer)  &  \\
					\hline
				\end{tabular}
		}		\end{center}
	\end{enumerate}
\end{prop}

\begin{proof} The simply connected Lie group $Sp(1)$ is  geometrically the three dimensional sphere $S^3$ in $\rl^4$ and is isomorphic to $SU(2)$. Its  Lie algebra $L$  	can  therefore be identified with the  tangent space to $S^3$ at $(1,0,0,0)$ and therefore to the  space of purely vector
	quaternions $\rl^3=\{x_1i+x_2j+x_3k: x_1, x_2, x_3\in \rl\}$. The Lie bracket on $L$ is 
	 \[ [X,Y]=XY-YX=( -\langle X, Y\rangle + X\times Y)- (-\langle Y, X\rangle +Y\times X)= 2 (X\times Y), \]
	 where $\times$ denotes the cross product on $\rl^3$ and $\langle\cdot,\cdot\rangle$ is the inner product. 
	This Lie algebra $L$ cannot have a subalgebra of dimension 2, since as soon as we have in the subalgebra two linearly independent vectors $u,v$, their Lie bracket $2(u\times v)$ is another vector linearly independent of $u$ and $v$. Therefore a polarizing subgroup $G$ of $Sp(1)$ can be of dimension 3, 1 or 0 (i.e. finite). 	The only Lie subalgebra of  $L$ of dimension 3 is $L$ itself, and the corresponding subgroup of $Sp(1)$ is $Sp(1)$, which is
	polarizing as the unitary group of $\hx$.

		Now let $G$ be a polarizing subgroup of $Sp(1)$ of dimension 1, and let $G^0$ be its identity component. Then there is a purely vector quaternion ${u}\in L$ with $\norm{u}=1$
		 such that $G^0=\{\exp(\theta {u}): \theta \in \rl\}$. Since $u^2=-1$, it follows that $\exp(\theta u)=\cos \theta+ (\sin \theta )u$.  Since $G^0$ is abelian, 
		  $G\not = G^0$. Let $z\in G\setminus G^0$, so that, since $G^0$ is a normal subgroup of $G$ we have $zG^0z^{-1}=G^0$.
		 Therefore for each $\theta\in \rl$ there is a $\phi\in \rl$ such that $z \exp(\theta u)z^{-1}= \exp(\phi u)$. By looking at the scalar and vector parts, we see that $\cos\theta=\cos \phi$ and $(\sin \theta) zuz^{-1}=(\sin \phi) u$. The first equation gives $\theta=\pm \phi \,\mathrm{  mod  }\, 2\pi$, which when substituted into the second gives $zuz^{-1}=\pm u$.
		  Recall that for a unit quaternion $z\in Sp(1)$, the map $\rl^3 =\Im \hx\ni w\mapsto zwz^{-1}$ is a rotation of $\rl^3$ whose axis is spanned by the vector part $\Im z\in \rl^3$ of $z$.
		  Therefore, if $zuz^{-1}=u$, then $z=\cos\alpha+ (\sin \alpha) u$ for some $\alpha\in \rl$, which means that $z\in G^0$ which contradicts $z\not \in G^0$. Therefore we must have 
		  $zuz^{-1}=-u$. This means that the rotation $w\mapsto zwz^{-1}$ of $\rl^3$ has for axis a unit vector $v\in \rl^3$ orthogonal to $u$ and has angle of rotation $\pi$, so that 
		  $z=\cos\left(\dfrac{\pi}{2}\right)+ \sin\left(\dfrac{\pi}{2}\right)v=v$. Therefore each element of  $G\setminus G^0$ is a unit vector of $\Im(\hx)$  orthogonal to $u$. If $v\in G\setminus G^0$, notice that each element of the coset $G^0v= \{ \left(\cos \theta\right)v+ (\sin \theta )uv:\theta\in \rl\}$ is a unit vector of $\Im(\hx)$ orthogonal to $u$, therefore
		  \[ G= \{ \cos \theta + (\sin \theta)u : 0\leq \theta \leq 2\pi \}\cup \{\left(\cos \phi\right) v + (\sin \phi)uv  : 0 \leq\phi \leq 2 \pi \}.\]
		  There is an element $A\in SO(3)$ which maps the ordered basis $\{u,v, uv=u\times v\}$ into the standard basis $\{i,j,k\}$. Thanks to the double covering $Sp(1)\to SO(3)$ then there is a $\gamma\in Sp(1)$  such that for each $w\in \rl^3=\Im(\hx)$ we have $Aw =\gamma w\gamma^{-1}$. It follows that $\gamma G \gamma^{-1}$ is the group of \eqref{eq-onedim}.
		
	Finite multiplicative subgroups of $\hx^\times$ 
	can be classified up to conjugacy by using the double cover $Sp(1)\to SO(3)$ and then using the well-known classical classification of the finite subgroups of rotations of three dimensional space (see, e.g., \cite[p.~33]{conway}).  Using this classification one obtains the  complete list of nonabelian finite subgroups of $\hx^\times$ in the table above, and therefore of finite polarizing subgroups of $\hx$.

\end{proof}

\textbf{Remarks:}
\begin{enumerate}[wide]
	\item 
	Special cases of the quaternionic polarization  formula have been obtained previously in \cite{kurepa} and also in the unpublished PhD thesis \cite[Theorem~5, p. 75]{jamisonthesis}.
	These results correspond to the group $2D_4$ in the last row of the table, which is the 8 element ``quaternion group" $\langle i, j \rangle = \{\pm 1, \pm i, \pm j, \pm k\}$ and in \cite{jamisonthesis}, a positive definite 
	inner product on a quaternionic vector space. This formula over the quaternion group  is also contained
	in \cite{giarruso}, which proves an analogous formula for Clifford numbers.
	
	\item The finite groups above are closely connected with regular polyhedra and other  highly symmetric solids in $\rl^3$, just as the complex polarization formula \eqref{eq-pol-c} is associated 
	to a regular $N$-gon, and the integral version \eqref{eq-pol-int} is associated to the circle. Recall that there is a covering map of two sheets $\Phi:Sp(1)\to SO(3)$ given by $\Phi(z)(x)=zxz^{-1}$ for $x\in \rl^3$ {identified with the purely vector quaternions}. Under this map, each of the polarizing groups above is mapped to a subgroup of $SO(3)$ which is the rotational symmetry group of a well-known three dimensional figure.
	For $Sp(1)$, we get $SO(3)$, which is the symmetry group of the unit sphere $S^2\subset \rl^3$. The group \eqref{eq-onedim} maps under $\Phi$ to the symmetry group of a cylinder.
	The finite groups in the table map under $\Phi$ to the symmetry groups of the Platonic solids (as recognized in the traditional names of these groups.) On the other hand, the convex hull in $\hx$ of the  groups $2A_4$ and $2A_5$ are themselves 4 dimensional regular polytopes 
	called a 24-cell and 600-cell respectively. 
	
\end{enumerate}

As a consequence of the classification of polarizing subgroups of the division algebras we have the following:

\begin{prop}\label{prop-finite}
	Let $(\Aa, *)$ be a polarizable algebra. Then there is a finite polarizing subgroup $G\subset \Gamma(\Aa,*)$.
\end{prop}
\begin{proof}
	By Theorem~\ref{thm-kappa}, it suffices to show that the $*$-direct sum $\bigoplus_{j}^N (M_{n_j}(\D_j),*)$ admits a finite polarizing subgroup. For this we need to show that 
	each matrix algebra $(M_{n_j}(\D_j),*)$ has a polarizing subgroup $G_j$ which is finite, since then the product $G_1\times \dots \times G_N$ is clearly a polarizing subgroup of the direct sum
	$\bigoplus_{j}^N (M_{n_j}(\D_j),*)$ . Now we saw above  that each of the division algebras $\rl, \cx,\hx$ has a finite polarizing subgroup (and infinitely many of them for $\cx$ and $\hx$.) It now follows by Proposition~\ref{prop-polargroupsMnA} that whenever $\D$ has a finite polarizing subgroup $G$, the matrix algebra $M_n(\D)$ has a finite polarizing subgroup $\Delta_G H$ of cardinality  $\abs{G}^n \abs{H}$. Therefore each $(M_{n_j}(\D_j),*)$  has a finite polarizing subgroup and the result follows. 
\end{proof}
\subsection{Group algebras}
Let $G$ be a finite group, and let $\Aa= \rl[G]$ be its real group algebra. Recall  that this consists of formal linear combinations
$\sum_{g\in G} a_g g$, where $a_g\in \rl$, with the natural $\rl$-linear structure and a bilinear, associative multiplication induced by the multiplication operation of the group $G$. We can define 
an involution $*$  on $\Aa$ by the $\rl$-linear extension of the inversion operation on the group $G$:
\[ \left(\sum_{g\in G} a_g g\right)^* =  \left(\sum_{g\in G} a_g g^{-1}\right),\]
which is an involution since $(gh)^*=(gh)^{-1}= h^{-1}g^{-1}=h^*g^*$. The identity $1_\Aa$ of the algebra $\Aa$ is clearly $1\cdot e_G$, where $e_G$ is the identity element of the group $G$.

Identifying $g\in G$ with the element $1\cdot g\in \rl[G]$ we see that  $G\subset \Gamma(\rl[G],*)$, since for $g\in G$, we have $gg^*=gg^{-1}=e_G$. It follows that $G$ is a compact subgroup of the unitary group of $\rl[G]$ which generates it as an algebra, and 
consequently $\rl[G]$ is polarizable, thanks to Proposition~\ref{prop-structure1}. The group
\[  \{(\pm 1)g: g\in G\}\subset \Gamma(\rl[G], *)\]
isomorphic to $\Z_2\times G$ contains $-1 \cdot e_g$, hence is a polarizing group in $\rl[G]$ by Proposition~\ref{prop-firstmoment}. Theorem~\ref{thm-kappa} in this case reduces to  the classical Maschke's theorem (semisimplicity of the group algebra). The polarization constant of $\rl[G]$ can be computed in terms of 
the representation theory of the group $G$. We introduce the following notation:
	\begin{itemize}
	\item $\irr(G)$ is the set of irreducible complex characters of $G$,
	\item $\deg \chi=\chi(e_G)$ is the degree of $\chi\in \irr(G)$, i.e. the  dimension of the vector space on which the representation affording $\chi$ is defined, and
	\item for $\chi\in \irr(G)$, $s(\chi)$ denotes the Frobenius-Schur indicator of  $\chi$. Recall that this is respectively $1, 0$ or $-1$ according 
	to  whether the representation affording $\chi$  is real, complex or quternionic. 
\end{itemize}

\begin{prop}
	Then
	\begin{equation}\label{eq-kappagroupring}
			 \kappa(\rl[G],*)=\frac{1}{\abs{G}}\sum_{g\in G} \left(\sum_{\chi \in \irr(G)}\frac{(\deg \chi)^2 \cdot\ol{\chi (g)}}{\deg \chi +s(\chi)}\right)\cdot g.
	\end{equation}
	
\end{prop}
Notice that if $\chi\in\irr(G)$ is afforded by a quaternionic representation, we must have $\deg \chi\geq 2$. Therefore in all cases we have 
$\deg \chi+s(\chi)\geq 1$. It is also not difficult to see that the coefficient of each $g\in G$ in the sum on the right hand side of \eqref{eq-kappagroupring} is real valued.
\begin{proof}

Let
\[ r_2(g)= \abs{ \{x\in G: x^2 =g\}}\]
be the number of square roots of the element $g$ in the group $G$. Then with respect to the inner product $(\phi,\psi)= \frac{1}{\abs{G}}\sum_{g\in G}\phi(g)\ol{\psi(g)}$ on the space of complex-valued functions on $G$, we have for $\chi\in \irr(G)$:
\[ (r_2, \chi)= \frac{1}{\abs{G}} \sum_{g\in G} r_2(g)\ol{\chi(g)}=  \frac{1}{\abs{G}} \sum_{g\in G} \ol{\chi(g^2)}= \ol{s(\chi)}=s(\chi),\]
where we use the representation $s(\chi)=\frac{1}{\abs{G}}\sum_{g\in G} {\chi(g^2)}$ and the fact that $s$ is real valued.
Since $r_2:G\to \cx$ is a class function, which corresponds to the fact (Proposition~\ref{prop-momentproperties})  that
$\mu_2(G)$ is in the center of $\rl[G]$, and  since $\irr(G)$ is an orthonormal basis of the class functions with respect to the inner product  $(\cdot,\cdot)$ defined above, we can write
\begin{equation}
	\label{eq-r2}
	 r_2 = \sum_{\chi\in \irr(G)}(r_2,\chi)\chi= \sum_{\chi\in \irr(G)}s(\chi)\chi.
\end{equation}
Since $G$ is a compact (even finite) subgroup of the algebra $\Aa=\rl[G]$ generating it as an algebra, the second moment of $\Aa$ can be 
computed using \eqref{eq-mu2inv}:
\begin{align*}
 \mu_2(\rl[G])&=\mu_2(G)=\frac{1}{ \abs{G}} \sum_{g\in G}g^2=\frac{1}{\abs{G}} \sum_{g\in G} r_2(g)g\\
 &=\frac{1}{\abs{G}} \sum_{g\in G} r_2(g^{-1})g \quad\text{ since $r_2(g)=r_2(g^{-1})$ }\\
  &=\frac{1}{\abs{G}} \sum_{g\in G}\left( \sum_{\chi\in \irr(G)}s(\chi)\chi(g^{-1})\right)g\quad \text{using \eqref{eq-r2}.}\\
  &=\sum_{\chi\in \irr(G)}\frac{s(\chi)}{\deg \chi} \cdot\epsilon_\chi,
\end{align*}
where  for $\chi \in \irr(G)$, we set
\begin{equation}\label{eq-epschi}
   \epsilon_\chi = \frac{\deg \chi}{\abs{G}}\sum_{g\in G} \chi(g^{-1}) g.
\end{equation}
It is known (see \cite[Chapter 2]{isaacs}) that the collection of elements $\{\epsilon_\chi: \chi \in \irr(G)\}\subset \rl[G]$  form a complete set of  primitive orthogonal idempotents in the commutative ring $Z(\rl[G])$, the center of $\rl[G]$, i.e. they satisfy
\[ \epsilon_\chi^2=\epsilon_\chi, \quad \epsilon_\chi\epsilon_\lambda =0 \quad\text{ for } \chi, \lambda \in \irr(G), \chi\not=\lambda,\quad\text{and   } \sum_{\chi \in\irr(G)} \epsilon_\chi = 1_{\rl[G]}.   \]
Therefore by \eqref{eq-kappa3} we have
\begin{align*}
\kappa(\rl[G],*)&= \left(1_{\rl[G]}+\mu_2(G)\right)^{-1}\\
&= \left(\sum_{\chi\in \irr(G)} \epsilon_\chi + \sum_{\chi\in \irr(G)}\frac{s(\chi)}{\deg \chi} \cdot\epsilon_\chi\right)^{-1}\\
&= \left(\sum_{\chi\in \irr(G)} \left( 1+\frac{s(\chi)}{\deg \chi} \right)\epsilon_\chi\right)^{-1}\\
&=\sum_{\chi\in \irr(G)}\frac{\deg \chi}{\deg \chi+s(\chi)}\cdot\epsilon_\chi.
\end{align*}
Substituting   $\epsilon_\chi$  from \eqref{eq-epschi} leads to the formula \eqref{eq-kappagroupring}.
\end{proof}

\subsection{Clifford algebras}  For nonnegative integers $p,q$, let $\mathsf{C}_{p,q}$ be the real Clifford algebra of signature $(p,q)$ (see \cite{spin}). Recall that this is a $*$-algebra over $\rl$  of dimension $2^{p+q}$ generated by the $p+q$ elements $\{e_i: 1\leq i \leq p+q\}$ satisfying the relations
\[
e_i^2=1  \quad \text{ for } 1\leq i \leq p, \quad
e_j^2=-1 \quad \text{ for } p\leq j \leq p+q, \quad
e_ie_j =-e_je_i \quad \text{ for } i\not =j.
\]
The involution of $\mathsf{C}_{p,q}$ is determined by $e_i^*= e_i, 1\leq i \leq p$ and $e_i^*=-e_i, p+1\leq i \leq p+q$.
If for $I=\setof{i_1 < \dots < i_r} \subseteq \setof{1 ,\dots , p+q}$ we set 
\[ e_I = e_{i_1}e_{i_2}\dots e_{i_r}\]
then $e_I^*= e_{i_r}^*e_{i_{r-1}}^*\dots e_{i_1}^* $.
It is easy to see that 
\[ G_{p,q} = \setof{ \pm e_I \st I = \setof{i_1 < \dots < i_r} \subseteq \setof{1 ,\dots , p+q}} \]
is a finite subgroup of $\Gamma(\mathsf{C}_{p,q},*)$ which generates $\mathsf{C}_{p,q}$. Therefore by part  (3) of Proposition~\ref{prop-structure1} the algebra $\mathsf{C}_{p,q}$ is polarizable.
We now find its polarization constant. 
\begin{prop}$ \displaystyle{\kappa(\mathsf{C}_{p,q}) = \dfrac{1}{1+{2^{-\frac{p+q-1}{2}}}\cos\allof{\frac{p-q -1}{4} \pi}}}$ 
\end{prop}

\begin{proof}
 Suppose that $I=  \setof{i_1 < \dots < i_k}\subseteq \setof{1 ,\dots , p+q}$, so that $\abs{I} = k$. Assume that  $\abs{I \cap \setof{1 ,\dots , p}} = \ell$. Then 
\begin{align*}
    e_I^2 &= e_{i_1} \cdots e_{i_\ell} e_{i_{\ell+1}} \cdots e_{i_k}\cdot e_{i_1} \cdots e_{i_\ell} e_{i_{\ell+1}} \cdots e_{i_k} \\
          &=(-1)^{k-1}\cdots (-1)^1 e_{i_1}^2 \cdots e_{i_\ell}^2 e_{i_{\ell+1}}^2 \cdots e_{i_k}^2\\
          &=(-1)^{\frac{k(k-1)}{2}}(-1)^{k-\ell}\\
          &=(-1)^{\frac{k(k+1)}{2}}(-1)^{\ell}.  
\end{align*}
For such an $I$ there are $\begin{pmatrix} p \\ \ell \end{pmatrix}$ possibilities for $I \cap \setof{1, \cdots, p}$ and $\begin{pmatrix} q \\ k- \ell \end{pmatrix}$ possibilities for $I \setminus \setof{1 , \cdots , p}$. Hence there are $\begin{pmatrix} p \\ \ell \end{pmatrix}\begin{pmatrix} q \\ k- \ell \end{pmatrix}$ such $I$ and 
\[ \mu_2(\mathsf{C}_{p,q}) = \mu_2(G_{p,q}) = \dfrac{1}{2^{p+q}} \sum_{I} e_I^2= \dfrac{1}{2^{p+q}} \sum_{k=0}^{p+q} (-1)^{\frac{k(k+1)}{2}} \sum_{\ell=0}^k \begin{pmatrix} p \\ \ell \end{pmatrix}\begin{pmatrix} q \\ k- \ell \end{pmatrix}(-1)^{\ell}.  \]
Next observe that $(-1)^{\frac{k(k+1)}{2}}$ is periodic in $k$ with period $4$. The equality 
$
    (-1)^{\frac{k(k+1)}{2}} = \Re( (1+i)i^k )
$
is easily verified by periodicity. From this we get 
\begin{align*}
    \mu_2(\mathsf{C}_{p,q}) &= \dfrac{1}{2^{p+q}} \Re \allof{(1+i)\sum_{k=0}^{p+q} i^k \sum_{\ell=0}^k \begin{pmatrix} p \\ \ell \end{pmatrix}\begin{pmatrix} q \\ k- \ell \end{pmatrix}(-1)^{\ell}} \\
    &=\dfrac{1}{2^{p+q}} \Re \allof{(1+i) \allof{\sum_{r=0}^p \begin{pmatrix} p \\ r \end{pmatrix} (-1)^r i^r} \allof{\sum_{s=0}^q \begin{pmatrix} q \\ s \end{pmatrix} i^s}} \\
    &= \dfrac{1}{2^{p+q}} \Re\allof{(1-i)^p(1+i)^{q+1}} \\
    &= \dfrac{2^{\frac{p+q+1}{2}}}{2^{p+q}} \Re\allof{e^{-p\pi i/4}e^{(q+1)\pi i / 4}} \\
    &= \dfrac{1}{2^{\frac{p+q-1}{2}}}\cos\allof{\frac{p-q -1}{4} \pi}.
\end{align*}
The result follows by \eqref{eq-kappa3}.
\end{proof}

For the algebras $\mathsf{C}_{0,q}$, the polarization identity corresponding to the group $G_{0,q} $ was found in \cite{giarruso}. The case $q=2$ (quaternions) was already done in \cite{jamisonthesis}.

\subsubsection{Matrix algebras} A simple special case of Theorem~\ref{thm-kappa} is that
 the matrix algebras $M_n(\D)$  where $D=\rl,\cx,\hx$ with the conjugate-transpose involution are polarizable, 
and the polarization constant is $\kappa = \frac{n\delta}{(n-1)\delta+2}I_n$, where $\delta= \dim_\rl \D$.
Here we note some further information about polarizing subgroups of $M_n(\rl)$ and $M_n(\cx)$ obtained by using the classical theorem of  Burnside that a subgroup of $GL_n(\cx)$ spans $M_n(\cx)$ over $\cx$ if and only if the natural representation of $G$ on $\cx^n$ by matrix multiplication is irreducible (see e.g. \cite{lang}):

\begin{prop}
\label{prop-MnRpolarG}
\begin{enumerate}
    \item 
A compact subgroup $G$ of $\Gamma(M_n(\rl))= O(n)$ is polarizing in $M_n(\rl)$ if and only if $G\not=\{I_n\}$ and the natural representation of $G$ on $\cx^n$ is irreducible.
\item A compact subgroup $G$ of $\Gamma(M_n(\C))= \mathrm{U}(n)$ is polarizing in $M_n(\C)$ if and only if $G\not =\{I_n\}$, the natural representation of $G$ on $\C^n$ is irreducible and $i I_n \in \mathrm{span}_\rl G$ where $I_n$ is the identity in $M_n(\C)$.   
\end{enumerate}

\end{prop}

\begin{proof}
Suppose that $G$ is polarizing in $M_n(\rl)$. Then $\mu_1(G)=0$ implies that $G$ is nontrivial. Since $G$ generates $M_n(\rl)$ as a real algebra it follows that $G$ generates $M_n(\cx)$ as a complex algebra. Thus Burnside's theorem implies that the natural action of $G$ on $\cx^n$ is irreducible.

Conversely, if $G$ acts irreducibly on $\cx^n$, then $G$ generates $M_n(\cx)$ as a complex algebra by Burnside's theorem. Since the matrices have real entries,  $G$ generates $M_n(\rl)$ as a real algebra. Then by part (b) of Proposition~\ref{prop-firstmoment}, $\mu_1(G) = 0$. Therefore $G$ is polarizing in $M_n(\rl)$.

Now suppose that $G$ is polarizing in $M_n(\cx)$. Again, $\mu_1(G) =0 $ implies that $G$ is nontrivial. Since $G$ generates $M_n(\cx)$ as a real algebra, i.e. $\mathrm{span}_\rl G = M_n(\cx)$, it follows that $i I_n \in \mathrm{span}_\rl G$.  It follows that $M_n(\cx) = \mathrm{span}_\rl G = \mathrm{span}_\cx G$, and so Burnside's theorem says that the natural action of $G$ on $\cx^n$ is irreducible.

Conversely, suppose that $G$ is nontrivial, the natural representation of $G$ on $\cx^n$ is irreducible, and $i I_n \in \mathrm{span}_\rl G$. Burnside's theorem yields that $G$ generates $M_n(\cx)$ as a complex algebra, i.e. $\mathrm{span}_\cx G = M_n(\cx)$. But as was noted above, the fact that $i I_n \in \mathrm{span}_\rl G$ implies that $\mathrm{span}_\rl G = \mathrm{span}_\cx G = M_n(\cx)$. Then again 
Proposition~\ref{prop-firstmoment} shows that 
$\mu_1(G)=0$. Therefore $G$ is polarizing in $M_n(\cx)$.
\end{proof}

As a result we have the following:

\begin{prop}
	The group $SO(n)$ (resp. $SU(n)$) is polarizing in $M_n(\rl)$ (resp. $M_n(\cx)$) if and only if $n\geq 3$.
\end{prop}
\begin{proof}
If $n=1,2$ it is clear that $SO(n)$ (resp. $SU(n)$) fails to span the algebra $M_n(\rl)$ (resp. $M_n(\cx)$): note that $SO(2)$ (resp. $SU(2)$) generates an algebra isomorphic to $\cx$ (resp. $\hx$). Also it is not difficult to show that the natural representation of $SO(n)$ on $\cx^n$ is irreducible using Schur's lemma. The same holds for $SU(n)$ as it contains $SO(n)$. To complete the proof we only need to verify that $iI_n\in \mathrm{span}_\rl(SU(n))$ for $n\geq 3$.

 If $n=4k$ for some $k \geq 1$, then in fact $i I_n \in SU(n)$. If $n = 4k+1$ for some $k \geq 1$, consider the diagonal matrices $D_\ell \in M_{4k+1}(\cx)$ with a $1$ in the $\ell$-th entry and an $i$ in each other entry. Each of these is in $SU(4k +1)$ and $D_\ell - D_\ell^{-1}$ is the diagonal matix with a $0$ in the $\ell$-th entry and $2i$ in each of the other entries. Thus
\[iI_{4k+1} = \dfrac{1}{8k} \sum_{\ell = 1}^{4k+1} (D_\ell - D_\ell^{-1}) \in \mathrm{span}_\rl SU(4k+1).\]
A similar argument can be used to show that $i I_{4k+3} \in \mathrm{span}_\rl SU(4k+3)$  for $k \geq 0$ by considering the diagonal matrices $D_\ell \in SU(4k+3)$ with a $-1$ in the $\ell$-th entry and an $i$ in each other entry. Thus
\[iI_{4k+3} = \dfrac{1}{2(4k+2)} \sum_{\ell = 1}^{4k+3} (D_\ell - D_\ell^{-1}) \in \mathrm{span}_\rl SU(4k+3).\]
For $n = 4k+2$, $k \geq 1$, a further modification is needed. In this case we consider the diagonal matrices $D_{k,\ell} \in SU(4k+2)$ with a $1$ in the $k$-th and $\ell$-th entries and an $i$ in each other entry for $1 \leq k < \ell \leq n$. Then there are $n(n-1)/2 =(4k+2)(4k+1)/2$ summands $D_{k,\ell} - D_{k,\ell}^{-1}$, and $n-1=4k+1$ of them have a $0$ in a given entry and a $2i$ in each other entry. Thus
\[iI_{4k+2} =\dfrac{1}{4k(4k+1)} \sum_{1 \leq k < \ell \leq n} (D_{k,\ell} - D_{k,\ell}^{-1}) \in \mathrm{span}_\rl SU(4k+2). \]
\end{proof}

\bibliographystyle{alpha}
\bibliography{polarization}
\end{document}